\documentclass[11pt]{paper}

\usepackage{amssymb,amsfonts,amsthm,amsmath}

\usepackage{latexsym,multicol,graphicx}

\usepackage{graphicx, verbatim}
\usepackage{tikz}
\usepackage{xcolor}
\usepackage{verbatim}

\usepackage{pict2e}

\usepackage[a4paper, hmargin=2.45cm, vmargin={3cm, 3cm}]{geometry}

\usepackage{fancyhdr}
\usepackage[title,titletoc,toc]{appendix}

\newtheorem{theorem}{Theorem}
\newtheorem{lemma}[theorem]{Lemma}

\newtheorem{corollary}[theorem]{Corollary}

\newtheorem{claim}[theorem]{Claim}

\renewcommand{\phi}{\varphi}

\renewcommand{\emptyset}{\varnothing}

\newcommand{\eps}{\varepsilon}

\newcommand{\la}{\lambda}

\newcommand{\cE}{\mathcal E}

\newcommand{\cA}{\mathcal A}

\newcommand{\bE}{\mathbb E}

\def\done{{1\hskip-2.5pt{\rm l}}}

\renewcommand{\le}{\leqslant}
\renewcommand{\geq}{\geqslant}

\newcommand{\bR}{\mathbb R}

\newcommand{\bZ}{\mathbb Z}

\newcommand{\bP}{\mathbb P}

\newcommand{\sm}{{\raise0.3ex\hbox{$\scriptstyle \setminus$}}}

\def\bprf{\begin{proof}}
\def\eprf{\end{proof}}

\def\Sym{{\mathfrak S}} 

\def\sgn{\operatorname{sgn}}

\def\Slk{{\Sym_k^{\tt loc}}}

\def\Slkgd{{\Sym_k^{{\tt loc,good}}}}
\def\Sgk{{\Sym_k^{\tt gl}}}
\def\BB{{\Sym_n^{\tt bad}}}
\def\<{\langle}
\def\>{\rangle}

\begin{document}

\title{Real zeroes of random polynomials, II \\
Descartes' rule of signs and anti-concentration
on the symmetric group}

\author{Ken S\"{o}ze\thanks{290W 232nd Str, Apt 4b, Bronx, NY 10463, USA;
\texttt{sozeken65@gmail.com}}}

\date{\today}

\maketitle

\begin{abstract}
In this sequel to \cite{part1},
we present a different approach to bounding the expected number of real zeroes of random polynomials
with real independent identically distributed coefficients or more generally, exchangeable coefficients.
We show that the mean number of real zeroes does not grow faster than the logarithm of the degree.
The main ingredients of our approach are Descartes' rule of signs
and a new anti-concentration inequality
for the symmetric group.
This paper can be read independently of part~I in this series.
\end{abstract}

\section{The result on the expected number of real zeroes}
In this part, which can be read independently from the first part~\cite{part1} (see the latter for a brief history of the problem), we will bound the expected number of
real zeroes of random polynomials with real independent identically distributed coefficients or, more generally, exchangeable coefficients
(the definition of exchangeability is recalled a few lines later).

For a non-zero polynomial $P$ and a subset $A\subseteq \bR$, let $N(A,P)$ denote the number of zeroes of $P$, counted with multiplicity,  that fall in $A$.
We write $N(P)$ for $N(\bR,P)$ and $N^{*}(P)$ for $N(\bR\!\setminus\!\{0\},P)$. Everywhere in the paper, $C,c,C',c'$ etc.,
denote positive numerical constants (not depending on any parameters). However, the values of these constants may change from line to line.
With this notation, we are ready to state our main theorem and the key lemmas.

\begin{theorem}\label{thm:lognbound} Let $u_{0},\ldots ,u_{n}$, $n\ge 2$, be real numbers, not all equal to zero. Let $\pi$ be a uniform random permutation of
$\{0,1,\ldots ,n\}$. Let $ P(x) = \sum_{k=0}^n u_{\pi(k)} x^k $. Then
\[
\bE[N^{*}(P)]\le C\log n\,.
\]
\end{theorem}

As an almost immediate corollary, we get a bound on the expected number of zeros for random polynomials with i.i.d.\! or,
more generally, exchangeable coefficients. Recall that random variables $\la_{1},\ldots ,\la_{n}$ are said to be exchangeable if the distribution of $(\la_{\pi(1)},\ldots
,\la_{\pi(n)})$
is the same as the distribution of $(\la_{1},\ldots ,\la_{n})$ for any $\pi\in \Sym_{n}$, where $\Sym_n$ is the group of permutations of $\{1,\ldots ,n\}$.
Note that if $\la_{k}$ are i.i.d., then they are exchangeable.

\begin{corollary}\label{corollaryiid} Let $P_{n}(x)=\la_{0}+\la_{1}x+\ldots +\la_{n}x^{n}$, $n\ge 2$.
\begin{enumerate}
\item If $\la_{0},\ldots ,\la_{n}$ are exchangeable random variables, then $\bE\big[N^{*}(P_{n}) \done_{\{P_{n}\not=0\}}\big]\le C\log n$.
\item If  $\la_{0},\ldots ,\la_{n}$ are i.i.d.\! with $p_{0}:=\bP\{\la_{0}=0\}<1$, then $\bE\big[N(P_{n}) \, \done_{\{P_{n}\not=0\}}\big]\le C\log n+\frac{p_{0}}{1-p_{0}}$.
\end{enumerate}
\end{corollary}
The reason for the indicator function $\done_{\{P_{n}\not=0\}}$ is that there may be a positive probability for all coefficients to vanish (in which case $N(P_{n})$ is not defined).
If the coefficients are i.i.d., then $\bP\{P_{n}\not=0\}=1-p_{0}^{n+1}$, showing that restricting to this event leaves out only a tiny part of the probability space,
provided that $n$ is large and $p_0$ is not too close to $1$.

\bprf Condition on the multi-set of values $\{\la_{0},\ldots ,\la_{n}\}$
(multi-set means that $\la_k$ need not be distinct). Conditional on this multi-set being equal to $\{u_{0},u_{1},\ldots ,u_{n}\}$, by exchangeability,
the random vector $(\la_{0},\ldots ,\la_{n})$ has the same distribution as $(u_{\pi(0)},\ldots ,u_{\pi(n)})$, where $\pi$ is a uniform random permutation of $\{0,1,\ldots ,n\}$.
On the event $P_{n}\not=0$, not all $u_{i}$ can equal zero, and hence Theorem~\ref{thm:lognbound} applies to give the first part of the corollary.

For a non-zero polynomial $P$, we have $N(P)=N^{*}(P)+N(\{0\},P)$. By the first part, $\bE\big[N^{*}(P_{n}) \, \done_{\{P_{n}\not=0\}}\big]\le C\log n$.
 If the coefficients are i.i.d., the probability that $N(\{0\},P_{n})=k$ is $p_{0}^{k}(1-p_{0})$ and hence
 \begin{align*}
 \bE\big[N(\{0\},P_{n})\, \done_{\{P_{n}\not=0\}}\big] \; = \; (1-p_{0})\sum_{k=0}^{n-1}kp_{0}^{k} \; \le \; \frac{p_{0}}{1-p_{0}}\,,
 \end{align*}
where the last inequality follows by extending the sum up to infinity.
 \eprf

\medskip
Our approach in this paper
is based on Descartes' rule of signs and on the following ``relative anti-concentration bound''.
\begin{lemma}\label{lemma3}
Let $(\xi_1, \ldots , \xi_k)$ be a  vector of exchangeable random variables having a distribution that is absolutely continuous with respect to Lebesgue measure on $\bR^{k}$. Then
\begin{equation}\label{eq:1st-sum}
\bP \Bigl\{ \Bigl| \sum_{j=1}^k j\xi_j \Bigr| \le\Bigl| \sum_{j=1}^k \xi_j \Bigr|\Bigr\} \le\frac{C}k \,,
\end{equation}
and
\begin{equation}\label{eq:2nd-sum}
\bP \Bigl\{ \Bigl| \sum_{j=1}^k (-1)^j j\xi_j \Bigr| \le\Bigl| \sum_{j=1}^k (-1)^j \xi_j \Bigr|\Bigr\} \le
\frac{C}{k} \,.
\end{equation}
\end{lemma}

\medskip In turn, Lemma~\ref{lemma3} will be deduced from an anti-concentration bound for linear forms
on the symmetric group $\Sym_n$. The following lemma may be considered the main technical result of this paper and potentially of interest beyond its application to proving our main
theorem.
\begin{lemma}\label{lem:mainanticoncentration}
Let $n\geq 2$ and let  $w_{1},\ldots ,w_{n}$ be real numbers such that $\sum_{i=1}^{n}w_{i}=0$ and $\sum_{i=1}^{n}w_{i}^{2}=1$. Let $\pi$ be a random permutation
uniformly distributed on $\Sym_{n}$. Then, for every $L\in \bR$, we have
\begin{align}
\label{eq:ineqatunitscale}
\bP\Bigl\{\Big| \sum_{i=1}^{n}w_{i}\pi (i) - Ln \Big| \le1\Bigr\}\le\frac{C}{n}\, e^{-c|L|}.
\end{align}
\end{lemma}
We remark that it is possible to strengthen the statement to have $\frac{C}n\, e^{-cL^{2}}$ on the right hand side
(see a brief discussion in Section~\ref{sec:proofofprobestimates}),
but we shall not need that improvement in this paper. Although the main application of this lemma is to prove Lemma~\ref{lemma3}, we shall also use it in several other smaller ways.
For this purpose, we record an easy corollary of Lemma~\ref{lem:mainanticoncentration}.
\begin{corollary}\label{cor:ineqatatom} Suppose $u_{1},\ldots, u_{n}$ are real numbers, not all equal. Let $\pi$ be a random permutation
uniformly distributed on $\Sym_{n}$. Then, for every $x\in \bR$, we have
\begin{align} \nonumber
\bP\Bigl\{\sum_{i=1}^{n}u_{i}\pi(i) =x\Bigr\}\le\frac{C}{n}.
\end{align}
\end{corollary}
\bprf
Set $w_{i}=(u_{i}-\bar{u})/\sigma$ where $\bar{u}=\frac{1}{n}\sum_{i=1}^{n}u_{i}$ and $\sigma^{2}=\sum_{i=1}^{n}(u_{i}-\bar{u})^{2}$. The assumption that
$u_{1},\ldots ,u_{n}$ are not all equal ensures that $\sigma>0$ and hence $w_{i}$ are well-defined, $\sum_{i=1}^{n}w_{i}=0$ and $\sum_{i=1}^{n}w_{i}^{2}=1$.
Now apply Lemma~\ref{lem:mainanticoncentration} with $Ln=(x-\tfrac12 n(n+1)\bar u)/\sigma$. The corollary follows.
\eprf

\section{Descartes' rule of signs and deduction of Theorem~\ref{thm:lognbound} from Lemma~\ref{lemma3}}

Recall that the number of sign-changes of a finite or infinite sequence $b=(b_{0},b_{1},\ldots)$ of real numbers is defined as the supremum of all $k$
for which there exist indices $0\le i_{0}<i_{1}<\ldots <i_{k}$ such that $b_{i_{j}}b_{i_{j-1}} <0$ for each $j=1, 2,\ldots , k$.
Let $S(b)$ denote the number of sign changes of $b$.
\begin{lemma}[Descartes' rule of signs~{\cite[Chapter~5, Problem~38]{PolyaSzego}}]
Let $ f(x)=b_{0}+b_{1}x+b_{2}x^{2}+\ldots$ be a non-zero power series with real coefficients and convergent in $(-R,R)$.
Let $N^{+}(f)$ be the number of zeroes (counted with multiplicity) of $f$ in the interval $(0,R)$.
Then,  $N^{+}(f)\le S(b)$.
\end{lemma}

\medskip
In the remaining part of this section, we prove Theorem~\ref{thm:lognbound} assuming that Lemma~\ref{lemma3} and Lemma~\ref{lem:mainanticoncentration} are true.

\subsection{Proof of Theorem~\ref{thm:lognbound} assuming Lemma~\ref{lemma3} and Corollary~\ref{cor:ineqatatom}}
Set $P(x)=\sum_{j=0}^n \la_j x^j$ with $\la_j = u_{\pi(j)}$ where $\pi$ is a uniform random
permutation.
Observe that $x^n P_n(1/x)=\la_{n}+\la_{n-1}x+\ldots +\la_{0}x^{n}$ is a random polynomial with the same distribution as $P_n$. Therefore, taking
$I=(0, 1)$ and $-I=(-1, 0)$, we can write
\begin{align}\label{eq:foursummands}
\bE [N^{*}(P_n)] =
2 \bE [N(I, P_n)] + 2 \bE [N(-I, P_n) ]
+  \bE [N(\{ 1 \}, P_n)] + \bE [N(\{ -1 \}, P_n)]\,,
\end{align}
where  $N(\{\pm 1\},P_{n})$ are the multiplicities of zeroes at $\pm 1$.

\medskip\noindent{\bf Bound for }$\bE[N(I,P_{n})]$: Consider the Taylor series with radius of convergence at least $1$:
\[
F(x) = \frac{P_n (x)}{1-x} = \sum_{k\geq 0} S_k x^k
\quad {\rm and} \quad
G(x) = \frac{P_n (x)}{(1-x)^2} = \sum_{k\geq 0} T_k x^k\,.
\]
Then
\[
S_k =
\begin{cases}
\la_0 + \la_1 + \ldots + \la_k & {\rm\ if\ } k \le  n, \\
S_n & {\rm\ if\ } k > n,
\end{cases}
\]
and
\[
T_k =
\begin{cases}
S_0 + S_1 + \ldots + S_k = (k+1)\la_0 + k\la_1 + \ldots + \la_k & {\rm\ if\ } k \le n, \\
T_n + (k-n)S_n & {\rm\ if\ } k > n\,.
\end{cases}
\]
Firstly, $N(I, P_n) = N(I, G) \le S((T_{k})_{k\ge 0})$, by Descartes' rule. Secondly, $S((T_{k})_{k\ge 0})\le 1+S(T_{0},T_{1},\ldots ,T_{n})$
since, beyond $n$, there could be at most one change of sign in the sequence $(T_k)_{k\ge n}$. Thus,
\begin{equation}\label{eq:upprboundonI}
\bE[N(I,P_{n})] \le 1+\bE[S(T_{0},\ldots ,T_{n})].
\end{equation}

Recall that $\la_{j}=u_{\pi(j)}$ where $\pi$ is a uniform random permutation. Assume without loss of generality that on the same probability space,
we have standard Gaussian random variables $Z_{j}$, $0\le j\le n$, that are independent among themselves and independent of $\pi$.
Set $\la_{j}^{\eps}=\la_{j}+\eps Z_{j}$ for  $\eps>0$.
Let $S^\eps_k$ and  $T^{\eps}_{k}$ be defined using $(\la_{j}^{\eps})_{0\le j\le n}$ exactly as $S_k$ and $T_{k}$ are defined in terms of $(\la_{j})_{j\le n}$. Then,
\[
S(T_{0}^{\eps},\ldots ,T_{n}^{\eps})\le \sum_{k=1}^{n}\done_{\{T_{k}^{\eps}T_{k-1}^{\eps}\le 0\}}.
\]
Since $T_{k}^{\eps}=T_{k-1}^{\eps}+S_{k}^{\eps}$, to have $T_{k}^{\eps}T_{k-1}^{\eps}\le 0$, it is necessary that $|T_{k}^{\eps}|\le |S_{k}^{\eps}|$.
Now, for any fixed $\eps>0$, the random variables $\la_{j}^{\eps}$, $0\le j\le n$, are exchangeable, and have an absolutely continuous distribution on $\bR^{n}$.
By conclusion~\eqref{eq:1st-sum} in Lemma~\ref{lemma3}, it immediately follows that
$\bP\{|T_{k}^{\eps}|\le |S_{k}^{\eps}|\}\le C/k$ and hence $\bE[S(T_{0}^{\eps},\ldots ,T_{n}^{\eps})]\le C\log n$. Observe that $C$ does not depend on $\eps$ (or anything else).

Since sign changes are defined by strict inequalities, we see that almost surely,
\[
S(T_{0},\ldots, T_{n})\le \liminf_{\eps\to 0} S(T_{0}^{\eps},\ldots ,T_{n}^{\eps}).
\]
and hence, by Fatou's lemma $\bE[S(T_{0},\ldots,T_{n})]\le C\log n$. Plugging back this conclusion into \eqref{eq:upprboundonI}, we get $\bE[N(I,P_{n})]\le C\log n$.

\medskip\noindent{\bf Bound for }$\bE[N(-I,P_{n})]$: Next we bound $\bE[N(-I,P_{n})]$. Replacing $x$ by $-x$, we have the analogue of \eqref{eq:upprboundonI}:
\begin{equation}\label{eq:upprboundon-I}
N(-I,P_{n}) \le 1+S(T_{0}',\ldots ,T_{n}')
\end{equation}
where  $S_{k}'=\sum_{j=0}^{k}(-1)^{j}\la_{j}$ and $T_{k}'=S_{0}'+\ldots +S_{k}'=\sum_{j=0}^{k}(k+1-j)(-1)^{j}\la_{j}$.

Exactly as before, we define $\la_{j}^{\eps}=\la_{j}+\eps Z_{j}$, where $Z_{j}$ are independent standard Gaussians that are also independent of
$\pi$. Define $S_{j}'^{\eps}$ and $T_{j}'^{\eps}$ in terms of $(\la_{j}^{\eps})_{j\le n}$ just as $S_{j}'$ and $T_{j}'$ are defined in terms of $(\la_{j})_{j\le n}$.
By the lower semi-continuity of sign changes, by letting $\eps$ decrease to zero, we may deduce that $\bE[S(T_{0}',\ldots ,T_{n}')]$ is bounded by $C\log n$ provided we  prove the same
bound for $\bE[S(T_{0}'^{\eps},\ldots ,T_{n}'^{\eps})]$. To do that, we write
\begin{align*}
\bE[S(T_{0}'^{\eps},\ldots ,T_{n}'^{\eps})]&\le \sum_{k=1}^{n}\bP\{T_{k}'^{\eps}T_{k-1}'^{\eps}\le 0\} \\
&\le \sum_{k=1}^{n}\bP\{|T_{k}'^{\eps}|\le |S_{k}'^{\eps}|\}.
\end{align*}
Now, use the bound \eqref{eq:2nd-sum} in Lemma~\ref{lemma3} to get $\bP\{|T_{k}'^{\eps}|\le |S_{k}'^{\eps}|\}\le C/k$. Using this bound in \eqref{eq:upprboundon-I}, we get the
inequality $\bE[N(-I,P_{n})]\le C\log n$.

\medskip\noindent{\bf Bound for }$\bE[N(\{\pm 1\},P_{n})]$: If $ N(\{1\},P_{n}) \geq 2$, we must have $P_{n}(1)=P_n'(1)=0$, and therefore $\sum_{k=0}^n (k+1) u_{\pi(k)} = 0$.
Obviously, this cannot happen if all $u_k$ are equal and not zero.
Then, Corollary~\ref{cor:ineqatatom} shows that this event has probability at most $C/n$. Therefore, $\bE[N(\{1\},P_{n})]\le C$ since the root at $1$ has multiplicity at most $n$.

Now we turn to the root at $-1$. If $N(\{-1\},P_{n})\ge 2$, then $P_{n}(-1)=P_{n}'(-1)=0$ and hence, $\sum_{k=1}^{n+1}(-1)^{k}k\la_{k-1}=0$. Using exchangeability, the probability of
this event  is the same as the probability of
\begin{equation}\label{eq:piequalsigma}
\sum_{k\in E_{n}}\pi(k)\la_{k-1}=\sum_{k\in O_{n}}\sigma(k)\la_{k-1}
\end{equation}
where $\pi$ and $\sigma$ are uniform random permutations of  $E_{n}=2\bZ\cap\{1,2,\ldots,n+1\}$ and $O_{n}=(2\bZ+1)\cap\{1,2\ldots ,n+1\}$ respectively, and $\pi, \sigma$ are
independent of each other and of $\la_{0},\ldots ,\la_{n}$. Fix the values of  $\la_{0},\ldots,\la_{n}$ and consider three  cases.

\smallskip\noindent{\bf Case 1}: Suppose $\la_{k-1}$, $k\in E_{n}$ are not all equal. In this case, fix $\sigma$ (i.e., condition on $\sigma$) so that the right hand side of
\eqref{eq:piequalsigma} is not random anymore. We may also write $\pi(k)=2\pi'(k/2)$ where $\pi'$ is a uniform random permutation of $\frac{1}{2}E_{n}=\{1,2,\ldots ,\lfloor
\frac{n+1}{2}\rfloor\}$. Apply Corollary~\ref{cor:ineqatatom} to $\pi'$ and conclude that the probability of the event in \eqref{eq:piequalsigma} is at most $C/n$.

\smallskip\noindent{\bf Case 2}: Suppose $\la_{k-1}$, $k\in E_{n}$ are all equal but $\la_{k-1}$, $k\in O_{n}$ are not all equal. Then we fix $\pi$ and write
$\sigma(k)=2\sigma'((k-1)/2)+1$ where $\sigma'$ is a uniform random permutation of $\frac{1}{2}(O_{n}-1)=\{0,1,\ldots , \lceil\frac{n-1}{2}\rceil\}$.
Apply Corollary~\ref{cor:ineqatatom} to $\sigma'$ and conclude that the probability of \eqref{eq:piequalsigma} is at most $C/n$.

\smallskip\noindent{\bf Case 3}: Suppose $\la_{k-1}=A$ for all $k\in E_{n}$ and $\la_{k-1}=B$ for all $k\in O_{n}$. If $A=B$, then $P_{n}$ is  a non-zero multiple of $1+t+t^{2}+\ldots
+t^{n}$
(recall that, by assumption, all $\la_k$ do not vanish simultaneously)
and $N(\{-1\},P_{n})\le 1$. Hence, we assume that $A\not=B$.
In this case, let $\tau$ be a uniform random permutation in $\{0,1,\ldots ,n\}$ and let $\la'_{k}=\la_{\tau(k)}$ so that $\la'$ has the same distribution as $\la$. The probability that
$\la'_{k}$ are equal for all $k\in E_{n}$ and equal for all  $k\in O_{n}$ is smaller than $e^{-cn}$ for some $c>0$. Outside this event of negligible probability, $\la'$ will fall into
one of the two cases considered above.

\medskip
Thus, in all cases, $\bP\{N(\{-1\},P_{n})\ge 2\}\le C/n$ and hence $\bE[N(\{-1\},P_{n})]\le C$.

\smallskip
In summary, we have shown that the first two terms on the right hand side of \eqref{eq:foursummands} are bounded by $C\log n$ and that the last two terms are bounded by $C$. Thus,
$\bE[N^{*}(P_{n})]\le C\log n$.
\hfill $\Box$

\medskip\noindent{\bf Remark}: The idea
of employing the sign-changes of the Taylor series of the function
$ (1-x)^{-1}P_n(x) $ was used already in the pioneering paper of Bloch-P\'olya~\cite{BP} and then discussed by Kac~\cite{kacsignalnoise}.
Combining this idea with the classical Kolmogorov-Rogozin inequality for the concentration function (see, for instance,~\cite{Esseen}),
one can get a cruder form of Theorem~\ref{thm:lognbound} with $\sqrt{n}$ in place of $\log n$.

\section{Lemma~\ref{lem:mainanticoncentration} yields Lemma~\ref{lemma3}}

The proof of Lemma~\ref{lemma3} is based on randomization over permutations acting on $( \xi_1,\, \ldots \,, \xi_k )$
combined with estimate~\eqref{eq:ineqatunitscale} in Lemma~\ref{lem:mainanticoncentration}.
Throughout, we say that $\pi\in \Sym_k$ acts on the tuple
$\xi=(\xi_1, \ldots , \xi_k)$ by setting $(\pi \xi)_j=\xi_{\pi(j)}$; we define similarly the action of $\pi$ on functions of $\xi$.
The proof of the first estimate in Lemma~\ref{lemma3}
employs the full permutation group $\Sym_k$ which keeps invariant the joint distribution of the sums
$\sum_{j=1}^k \xi_j$ and $\sum_{j=1}^k j \xi_j$.
This is no longer possible when dealing with the sums
$\sum_{j=1}^k (-1)^j \xi_j$, $ \sum_{j=1}^k (-1)^j j\xi_j$,
and we are forced to use subgroups of the permutation group $\Sym_k$.
Which subgroup to use depends on whether $k$ is odd or even, and we distinguish between these cases in what follows.

The proof of the first estimate~\eqref{eq:1st-sum} in Lemma~\ref{lemma3} is
significantly simpler than that of the second estimate~\eqref{eq:2nd-sum}.
The reader interested only in the case of symmetrically distributed i.i.d.s may skip the proof of
\eqref{eq:2nd-sum}, which is contained in Sections
\ref{sec-3.3} and \ref{sec-3.4}.

\subsection{A corollary to Lemma~\ref{lem:mainanticoncentration}}
We start with a straightforward corollary to
Lemma~\ref{lem:mainanticoncentration}, which
may be interesting on its own.
\begin{lemma}\label{cor:lemtosheppsconjecture}
Let $n\geq 2$ and let $u_{1},\ldots ,u_{n}$ be  real numbers, not all equal to zero. Let $\pi$ be a uniform random permutation in $\Sym_{n}$. Then
$$\bP\Bigl\{\Big|\sum_{i=1}^{n}u_{i}\pi (i) \Big| \le\Big| \sum_{i=1}^{n}u_{i}\Big| \Bigr\}\le\frac{C}{n}.$$
\end{lemma}

\noindent{\em Proof of Lemma~\ref{cor:lemtosheppsconjecture}}:
If $u_{i}$s are all equal (and hence non-zero), then the probability in the statement  is zero and
there is nothing to prove. Otherwise, write $u_{i}=\bar{u}+\sigma w_{i}$, where $\bar{u}$ is the mean
of $u_{1},\ldots ,u_{n}$,  and
$\sigma^2=\sum_{i=1}^n (u_i-\bar{u})^2$.
Then, $w_{1},\ldots ,w_{n}$ satisfy
the hypotheses of Lemma~\ref{lem:mainanticoncentration}. Assume without loss of generality that $\bar{u}\geq 0$.
We want to get a bound on
\begin{align*}
\bP\Bigl\{\Big|\sum_{i=1}^{n}u_{i}\pi (i) \Big| \le\Big| \sum_{i=1}^{n}u_{i}\Big| \Bigr\} &=
\bP\Bigl\{ \Big|\sum_{i=1}^{n}w_{i}\pi (i)  + \frac{n(n+1)}{2}\frac{\bar{u}}{\sigma}\Big|\le n\frac{\bar{u}}{\sigma}   \Bigr\} \\
&=\bP\Bigl\{ \sum_{i=1}^{n}w_{i}\pi (i) \in \Bigl[ \frac{n(n-1)}{2}\frac{\bar{u}}{\sigma},\frac{n(n+3)}{2}\frac{\bar{u}}{\sigma}\Bigr] \Bigr\}.
\end{align*}
We cover the interval
$$\left[ \frac{n(n-1)}{2}\,\frac{\bar{u}}{\sigma}, \frac{n(n+3)}{2}\,\frac{\bar{u}}{\sigma}\right]$$
by $\lceil \frac{n\bar u}{\sigma} \rceil$ intervals of length $2$, and apply \eqref{eq:ineqatunitscale} to each subinterval
(the value of $L$ is different for different intervals, but, in any case, $ |L|\geq c\,\tfrac{n\bar u}{\sigma}$).
We get
\begin{align}\label{eq:cor6ubd}
\bP\Bigl\{\Big| \sum_{i=1}^{n}u_{i}\pi (i) \Big| \le\Big| \sum_{i=1}^{n}u_{i}\Big| \Bigr\} &\le\frac{C}{n}\,e^{-c\frac{n\bar{u}}{\sigma}}\,\left\lceil
\frac{n\bar{u}}{\sigma} \right\rceil.
\end{align}
Note that $\lceil x\rceil e^{-cx}$ is bounded by $\frac{2}{c}\vee 1$ on $[0,\infty)$
(for $x\le1$ the bound is $1$ while for $x>1$ we bound $\lceil x\rceil$ by $2x$ and use that $\max\limits_{t>0} \ te^{-t}\le1$). Thus
\[
\frac{C}{n}\, e^{-c\frac{n\bar{u}}{\sigma}}\, \left\lceil \frac{n\bar{u}}{\sigma} \right\rceil \le\frac{C}{n}\left(\frac{2}{c}\vee 1\right) \le\frac{C'}n\,.
\]
This completes the proof of Lemma~\ref{cor:lemtosheppsconjecture}. \hfill $\Box$

\subsection{The first estimate in Lemma~\ref{lemma3}}
Here, we use Lemma~\ref{cor:lemtosheppsconjecture} to deduce \eqref{eq:1st-sum}.
Let $A$ be the set $\{ \xi_1, \ldots , \xi_k\}$ (note that $\xi_{k}$ are distinct and non-zero with probability $1$).
Conditional on $A = \{u_1, \ldots , u_k\}$, the
tuple $( \xi_1, \ldots , \xi_k )$ has the same distribution as $( u_{\pi (1)}, \ldots , u_{\pi(k)} )$, where $\pi$ is
uniformly distributed in $\Sym_{k}$.
Lemma~\ref{cor:lemtosheppsconjecture}
applies (as $\sum_{i} u_{\pi(i)} i$ has the same distribution as $\sum_{i}u_{i}\pi (i)$) to show that
$$
\bP\Bigl\{ \Bigl| \sum_{j=1}^k j\xi_j \Bigr| \le\Bigl| \sum_{j=1}^k \xi_j \Bigr|\;
\Big| \; A=\{u_{1},\ldots ,u_{k}\}\Bigr\} \le\frac{C}{k}\,.
$$
Since this holds for every realization of $A$,  we get
\[
\bP \Bigl\{ \Bigl| \sum_{j=1}^k j\xi_j \Bigr| \le\Bigl| \sum_{j=1}^k \xi_j \Bigr| \Bigr\}
\le\frac{C}{k},
\]
completing the proof. \hfill $\Box$

\subsection{The second estimate in Lemma~\ref{lemma3}, the odd case}
\label{sec-3.3}

Let $k=2m-1$ and define
\[
S_{\tt e}=\sum_{j=1}^{m-1}\xi_{2j}, \quad S_{\tt o}=\sum_{j=1}^{m}\xi_{2j-1}, \quad S=S_{\tt e}-S_{\tt o}\,.
\]
Set
$ \xi_{2j}'=\xi_{2j}-\tfrac1{m-1}S_{\tt e}$, $\xi_{2j-1}'=\xi_{2j-1}-\tfrac1m S_{\tt o}$,
\[
T_{\tt e} = \sum_{j=1}^{m-1} j\xi_{2j}', \quad T_{\tt o}=\sum_{j=1}^m j \xi_{2j-1}',
\quad  T=T_{\tt e}-T_{\tt o}\,.
\]
Then
\[
\sum_{j=1}^{2m-1} (-1)^j \xi_j = S_{\tt e} - S_{\tt o} = S\,,
\]
and
\begin{align*}
\sum_{j=1}^{2m-1} (-1)^j j \xi_j 
&= 2 \sum_{j=1}^{m-1} j \xi_{2j} - 2 \sum_{j=1}^m j \xi_{2j-1} + \sum_{j=1}^m \xi_{2j-1} \\
&= 2 \sum_{j=1}^{m-1} j \xi'_{2j} + 2\, \frac1{m-1}\, S_{\tt e} \cdot \frac{m(m-1)}2
- 2 \sum_{j=1}^m j \xi'_{2j-1} - 2\, \frac1{m}\, S_{\tt o} \cdot \frac{m(m+1)}2 + S_{\tt o} \\
&= 2T_{\tt e} - 2T_{\tt o} + mS_{\tt e} - m S_{\tt o}   \\
&= 2 \Bigl[ T + \frac{m}2\, S \Bigr]\,.
\end{align*}
Thus, we aim at proving the existence of a numerical constant $C$ so that
	\begin{equation}
		\label{eq-121014b}
		\bP\Bigl\{|T+\frac{m}{2}\,S|\le\frac12\, |S|\Bigr\}
		\le\frac{C}{m}.
	\end{equation}

\subsubsection{The subgroups $\Sym_{k}^{\tt e}$ and $\Sym_{k}^{\tt o}$ of $\Sym_{k}$}
	
	Let $\Sym_{k}^{\tt e}$ denote the subgroup of $\Sym_k$ that includes
	those permutations that involve
	only the even indices $\{2j\}_{j=1}^{m-1}$. Similarly,
	let $\Sym_{k}^{\tt o}$ denote the subgroup of $\Sym_k$ that includes
	those permutations that involve
	only the odd indices $\{2j-1\}_{j=1}^{m}$. Let
	$\widehat \Sym_k$ denote the subgroup of $\Sym_k$ consisting
	of permutations $\pi=\pi_{\tt e}\circ\pi_{\tt o}$ where $\pi_{\tt e}\in \Sym_k^{\tt e}$ and
	$\pi_{\tt o}\in \Sym_k^{\tt o}$ ($\widehat \Sym_k$ is a subgroup because the subgroups $\Sym_{k}^{\tt e}$ and
$\Sym_{k}^{\tt o}$ commute). Note that $S_{\tt e}$ and $S_{\tt o}$, and hence $S$, are
	invariant under the action of any $\pi\in \widehat\Sym_k$ on $\xi$;
	in particular, $\pi=\pi_{\tt e}\circ \pi_{\tt o}\in \widehat \Sym_k$ acts on
	$\xi'=\{\xi_{j}'\}_{j=1}^k$ again as a permutation.
    Instead of considering $\pi$ drawn uniformly from $\Sym_k$, we
	consider $\pi_{\tt e}$ and $\pi_{\tt o}$ drawn uniformly from
	$\Sym_k^{\tt e}$ and $\Sym_k^{\tt o}$, respectively, and take
	$\pi=\pi_{\tt e}\circ \pi_{\tt o}$, writing
	$T_{\tt o}^\pi=\pi_{\tt o} T_{\tt o}$, $T_{\tt e}^\pi=\pi_{\tt e} T_{\tt e}$ and
	$T^\pi=T_{\tt e}^\pi-T_{\tt o}^\pi$. Proving \eqref{eq-121014b}
	then reduces to proving that for any fixed tuple $\xi$ with distinct entries,
	\begin{equation}
		\label{eq-121014c}
		\bP^\xi\Bigl\{|T^\pi+\frac{m}{2}\,S|\le\frac12\, |S|\Bigr\}
		\le\frac{C}{m},
	\end{equation}
	where $\bP^\xi$ denotes averaging with respect to
	the product of uniform
	measures on
	$\Sym_k^{\tt e}$ and $\Sym_k^{\tt o}$, with $\xi$ fixed\footnote{
  Here, as well as in the case $k=2m$, we use the following observation. Let $f\colon \bR^k \to \bR$ be a Borel
  function (in the current instance, $f=\frac12 |S| -|T+\frac{m}2S|$) and let $\xi=(\xi_1, \ldots , \xi_k)$ be a vector of
  exchangeable random variables. Then, for any subset $\widetilde\Sym\subset\Sym_k$ and any probability distribution
  $\widetilde\bP$ on $\widetilde\Sym$, one has
  \[
  \bP\bigl\{ f(\xi) \ge 0 \bigr\} \le \sup_{x_1, \ldots , x_k}\,\widetilde\bP \bigl\{ \sigma\in\widetilde\Sym\colon f(x_{\sigma_1}, \ldots , x_{\sigma_k}) \ge 0\bigr\}\,.
  \]
  }.
Henceforth, we assume that $S\not=0$, which happens almost surely because of the assumption that $(\xi_{1},\ldots,\xi_{k})$ has an absolutely continuous  distribution.

\subsubsection{Proof of estimate~\eqref{eq:2nd-sum} in the case $k=2m-1$}
	Introduce the partition of $\bR$ determined by
	$$ I_j=[ (j-1/2)S, (j+1/2)S), \quad j\in \bZ\,.$$
		We have then
		\begin{eqnarray*}
\bP^\xi\Bigl\{ \bigl| T^\pi+\frac{m}{2}S \bigr|\le\frac12\, |S| \Bigr\} &= &
\bP^\xi \Bigl\{ \bigl| T_{\tt e}^\pi - T_{\tt o}^\pi + \frac{m}{2}S \bigr| \le\frac12\, |S| \Bigr\} \\ \\
&\le&
\sum_{\stackrel{j_{\tt e},j_{\tt o}\in \bZ}{|j_{\tt e}-j_{\tt o}+\frac{m}{2}|\le4}}
\bP^\xi\bigl\{ T_{\tt e}^\pi\in I_{j_{\tt e}},
T_{\tt o}^\pi\in I_{j_{\tt o}} \bigr\} \\
&= &
\sum_{\stackrel{j_{\tt e},j_{\tt o}\in \bZ}{|j_{\tt e}-j_{\tt o}+\frac{m}{2}|\le4}}
\bP^\xi\bigl\{T_{\tt e}^\pi\in I_{j_{\tt e}}\bigr\}\,
\bP^\xi\bigl\{T_{\tt o}^\pi\in I_{j_{\tt o}}\bigr\} \,.
\end{eqnarray*}
Note that in the range of summation of the last expression we have that
$\max(|j_{\tt o}|,|j_{\tt e}|)\geq m/5$ if $m>40$. Assuming this is the case, and using
that
$$\sum_{j_{\tt e}} \bP^\xi\bigl\{T_{\tt e}^\pi\in I_{j_{\tt e}}\bigr\}=
\sum_{j_{\tt o}} \bP^\xi\bigl\{T_{\tt o}^\pi\in I_{j_{\tt o}}\bigr\}=1,$$ we obtain that
\begin{equation}
	\label{eq-121014d}
	\bP^\xi\Bigl\{|T^\pi+\frac{m}{2}S|\le\frac12\, |S| \Bigr\}\le\sup_{|j|\geq m/5}
\bP^\xi\bigl\{ T_{\tt e}^\pi\in I_{j} \bigr\}+
\sup_{|j|\geq m/5}
\bP^\xi \bigl\{ T_{\tt o}^\pi\in I_{j}\bigr\}\,.
\end{equation}
We now apply estimate~\eqref{eq:ineqatunitscale} in Lemma~\ref{lem:mainanticoncentration}.
The argument is the same for either $T_{\tt e}$ or $T_{\tt o}$, so for concreteness set
\[
\beta^2 =  \sum_{j=1}^{m-1} ( \xi'_{2j} )^2
\]
and consider the term in the right-hand side of~\eqref{eq-121014d} involving
$T_{\tt e}$. Set $w_j=\xi'_{2j}/\beta$ so that $\sum_j w_j =0$ and $\sum_j w_j^2 = 1$.
Then, for any $j$,
\[
\bP^\xi\bigl\{ T_{\tt e}^\pi \in I_j \bigr\} \le
\bP\Bigl\{ \Bigl| \sum_{j=1}^{m-1} w_j \pi_{\tt e}(j) - \frac{jS}{\beta} \Bigr|  \le \frac{|S|}{\beta} \Bigr\}\,.
\]
By Lemma~\ref{lem:mainanticoncentration}, the RHS does not exceed
$ \tfrac{C}{m}\, e^{-\frac{cj|S|}{\beta m}} \, \bigl\lceil \tfrac{|S|}{\beta} \bigr\rceil$.
Since we are interested only in $j$s with $|j|\ge m/5$, the latter expression is bounded
by $ \tfrac{C}{m}\, e^{-\frac{c|S|}{\beta}} \, \bigl\lceil \tfrac{|S|}{\beta} \bigr\rceil$.
Since $\lceil x \rceil e^{-cx}$ is bounded above, we get the bound $C/m$.

			The same argument applies with $T_{\tt o}^\pi$
			replacing $T_{\tt e}^\pi$. Substituting in
	\eqref{eq-121014d} completes the proof of the lemma when $k$ is odd.
	\qed

\subsection{The second estimate in Lemma~\ref{lemma3}, the even case}
\label{sec-3.4}

Let $k=2m$. We need to show that
\[
\bP\Bigl\{ \Bigl| \sum_{j=1}^{2m} (-1)^j j \xi_j \Bigr| \le \Bigl| \sum_{j=1}^{2m} (-1)^j \xi_j \Bigr|  \Bigr\} \le \frac{C}m\,.
\]
Put
\[
S=\sum_{j=1}^{2m} (-1)^j \xi_j = \sum_{j=1}^m (\xi_{2j} - \xi_{2j-1})
\]
and set  $\eta_j=\xi_{2j}-\xi_{2j-1}$, noting that
	$S=\sum_{j=1}^m \eta_j$. As in the odd case, we center $\eta_j$ by
	introducing $\eta_j'=\eta_j-\frac1{m}S$. Then
\begin{align*}
\sum_{j=1}^{2m} (-1)^j j \xi_j &= \sum_{j=1}^m (2j)\xi_{2j} - \sum_{j=1}^m (2j-1)\xi_{2j-1}
= 2\sum_{j=1}^m j\eta_j + \sum_{j=1}^m \xi_{2j-1} \\
&= 2\sum_{j=1}^m j\eta_j' + (m+1)S + \sum_{j=1}^m \xi_{2j-1}
= 2\sum_{j=1}^m j\eta_j' + m\Lambda
\end{align*}
with
\begin{equation}\label{eq:Lambda}
\Lambda = \Bigl( 1+ \frac1m \Bigr) S +\frac1{m} \sum_{j=1}^m \xi_{2j-1}\,.
\end{equation}
Thus, we need to estimate
\[
\bP\Bigl\{ \Bigl| 2\sum_{j=1}^m j \eta_j' +m\Lambda \Bigr| \le \bigl| S \bigr| \Bigr\}\,.
\]

\subsubsection{A randomization over local and global permutations}

We introduce two subgroups of $\Sym_{k}$. The first, which we
refer to as \textit{local} permutations,
swaps the entries of the pairs $(\xi_{2j-1}, \xi_{2j})$. This subgroup is
generated by the transpositions
$\tau_j$, $j=1,\ldots,m$, which map
$(2j-1,2j)\to (2j,2j-1)$. Writing $\Slk$ for the subgroup of local
permutations, we note that $|\Slk|=2^m$.

The second subgroup of $\Sym_k$ that we employ, which we refer
to as \textit{global} permutations, swaps the whole pairs. This subgroup
is generated by the
permutations $\theta_{jj'}$, $1\le j<j'\le m$, which map
\[ (1,2,\ldots,2j-1,2j,\ldots,2j'-1,2j',\ldots,k-1,k)\mapsto
(1,2,\ldots,2j'-1,2j',\ldots,2j-1,2j,\ldots,k-1,k)\,.\]
For instance, if we originally had the tuple
$(\xi_1, \xi_2, \xi_3, \xi_4, \xi_5, \xi_6)$, we can get something
like $(\xi_2, \xi_1, \xi_3, \xi_4, \xi_6, \xi_5)$ after some local
permutation, and then $(\xi_3, \xi_4, \xi_6, \xi_5, \xi_2, \xi_1)$ after
a global permutation. Writing $\Sgk$ for the subgroup of global permutations, we note that
$|\Sgk|=m!$.
We  will consider
in what follows permutations $\pi$ from $\Sym_k$ that decompose as
$\pi=\pi^{\tt g}\circ \pi^{\tt l}$ with $\pi^{\tt l}\in \Slk$ and $\pi^{\tt g}\in \Sgk$
and randomize over $\pi^{\tt g}$ and $\pi^{\tt l}$ with $\pi^{\tt g}$ and $\pi^{\tt l}$ independent
and uniformly distributed over $\Sgk$ and $\Slk$. Note that
unlike odd and even permutations considered in case $k$ is odd,
local and global permutations do not in general commute.
We note that
\begin{itemize}
\item
The quantities $S$ and $\Lambda$ are invariant under the action of $\Sgk$.
In what follows, the explicit form~\eqref{eq:Lambda} of $\Lambda$ will be irrelevant, only the
$\Sgk$-invariance will matter.
\item
Random choice of $\pi^{\tt l}$ is equivalent to placing independent random signs in from of $\eta_1$,
\ldots , $\eta_m$.
\end{itemize}

We fix $\xi$ (and hence, $\eta$), and denote by $\bP^\xi$ the law
of $(\pi^{\tt g}\circ\pi^{\tt l}) \xi$ conditioned on $\xi$. We need to
show that
\[
\bP^\xi\Bigl\{ \Bigl| 2\sum_{j=1}^m \pi^{\tt g}(j)
(\pi^{\tt l}\eta)'_j + m\Lambda_{\pi^{\tt l}}\Bigr|
\le \bigl| S_{\pi^{\tt l}} \bigl|\Bigr\} \le \frac{C}m\,.
\]
Here, we put $S_{\pi^{\tt l}} = \pi^{\tt l}S$, $\Lambda_{\pi^{\tt l}} = \pi^{\tt l}\Lambda$, and, abusing notation,
we denote by $\pi^{\tt g}(j)$ the image of $j\!\in\!\{1, 2,\, \ldots \, , m\}$ under $\pi^{\tt g}$ viewed as a permutation on $m$ letters.

\subsubsection{Good and bad local permutations}

Put
\[
\beta_{\pi^{\tt l}}^2 = \sum_{j=1}^m \bigl( (\pi^{\tt l}\eta)'_j \bigr)^2\,, \quad L_{\pi^{\tt l}} =
\frac{\Lambda_{\pi^{\tt l}}}{2\beta_{\pi^{\tt l}}}\,.
\]
We need to estimate
\[
\bP^\xi\Bigl\{ \Bigl| \frac1{\beta_{\pi^{\tt l}}} \sum_{j=1}^m \pi^{\tt g}(j)
(\pi^{\tt l}\eta)'_j + m L_{\pi^{\tt l}}\Bigr|
\le \frac{|S_{\pi^{\tt l}}|}{2\beta_{\pi^{\tt l}}}\Bigr\}
\]
We wish to apply about $|S_{\pi^{\tt l}}|/\beta_{\pi^{\tt l}}$ times Lemma~\ref{lem:mainanticoncentration} with $w_j = (\pi^{\tt l}\eta)'_j/\beta_{\pi^{\tt l}}$
(this is a point where we use the randomization over $\Sgk$).
An obstacle is that, for some $\pi^{\tt l}\in\Slk$, the quantity $\beta_{\pi^{\tt l}}$ can be much smaller than $|S_{\pi^{\tt l}}|$.
The randomization over $\Slk$ will help us to circumvent this obstacle.

Put
\[
B^2 = \sum_{j=1}^m \eta_j^2\,,
\]
and note that this quantity is both $\Sgk$- and $\Slk$-invariant. The next two claims show that outside of a tiny
part of all local permutations, $ \beta_{\pi^{\tt l}} $ is comparable with $B$.

Given a tuple $\eta=(\eta_1, \ldots , \eta_m)$
	we call a permutation $\pi^{\tt l}\in \Slk$  \textit{good}
	if
	\[ \frac{1}{4} m \le|\{j: (\pi^{\tt l} \eta)_j>0\}| \le \frac34 m \] and
	\textit{bad} otherwise. Let $\Slkgd$ denote the subset of
	$\Slk$ consisting of good permutations.
	\begin{claim}
		\label{lem-Bbeta}
		For all $\pi^{\tt l}\in \Slkgd$, it holds that $\frac15 B^2\le \beta_{\pi^{\tt l}}^2 \le B^2$.
	\end{claim}
\bprf[Proof of Claim~\ref{lem-Bbeta}]
Since
\[
\beta_{\pi^{\tt l}}^2=\sum_{j=1}^m {(\pi^{\tt l}\eta)'_j}^2=
\sum_{j=1}^m \bigl( (\pi^{\tt l}\eta)_j - \frac1m S_{\pi^{\tt l}} \bigr)^2 = \sum_{j=1}^m \eta_j^2
- \frac1m S_{\pi^{\tt l}}^2 =
B^2 - \frac1m S_{\pi^{\tt l}} ^2\,,
\]
the upper bound $\beta_{\pi^{\tt l}}^2\le B^2$ is immediate.
It remains to prove the claimed lower bound. Note however that
\[
	\beta_{\pi^{\tt l}}^2=
\sum_{j=1}^m \bigl( (\pi^{\tt l}\eta)_j - \frac1{m}S_{\pi^{\tt l}} \bigr)^2
\geq \sum_{j\colon \sgn ((\pi^{\tt l}\eta)_j)=-\sgn (S_{\pi^{\tt l}})}\bigl( \frac1{m}S_{\pi^{\tt l}} \bigr)^2\,.
\]
Since $\pi^{\tt l}\in \Slkgd$, the sum on the right-side
is over at least $\frac14 m$ terms,
and thus we obtain that
	$$\beta_{\pi^{\tt l}}^2\geq \frac14\, m\,\bigl( \frac1{m}S_{\pi^{\tt l}} \bigr)^2 =
	\frac14\, (B^2-\beta_{\pi^{\tt l}}^2)\,.$$
	The claim follows.
	\eprf

Thus, for good local permutations $\pi^{\tt l}$, the parameter $ \beta_{\pi^{\tt l}}^2 $
is controlled by $ B^2 $. The next claim asserts that most of local permutations are good:
	
	\begin{claim}
		\label{lem-good}
		Let $(\xi_1,\ldots,\xi_{2m})$ be an exchangeable random vector having a distribution that is absolutely continuous to Lebesgue measure on $\bR^{k}$.
        Let $\eta_j = \xi_{2j}-\xi_{2j-1}$, $1\le j \le m$. Then
		$$\bP\bigl\{ \pi^{\tt l} \not\in \Slkgd \bigr\} \le e^{-c m}\,.$$
	\end{claim}

\bprf[Proof of Claim~\ref{lem-good}]
	Let $N^{\pi^{\tt l}}_\eta=|\{j: (\pi^{\tt l}\eta)_j>0\}|$.
	Note that
	under $\pi^{\tt l}$,
	the signs $\{\sgn(\pi^{\tt l} \eta)_j\}_{1\le j\le m}$
	are i.i.d. zero mean Bernoulli random variables taking the values
	$\{-1,1\}$.
	Letting $\{\eps_i\}_{i\geq 1}$ denote i.i.d random variables
	taking the values $\{0,1\}$ with equal probability, and $ \bP^\eta $ denote
    expectation with respect to $\pi^{\tt l}$,
    we have that
	$$\bP^\eta \Bigl\{ N^{\pi^{\tt l}}_\eta\not\in \Bigr[\tfrac14 m,\tfrac34 m\Bigr] \Bigr\}
	\le	2\bP \Bigl\{ \sum_{1 \le i \le m} \eps_i<\tfrac14 m \Bigr\}
	\le e^{-c m}\,,$$
	where the last inequality follows the classical Bernstein-Hoeffding inequality (which we will recall in Section~\ref{subsubsect:BH}
below). This proves the claim.
\eprf

\subsubsection{Proof of estimate~\eqref{eq:2nd-sum} in the case $k=2m$}

Given a good local permutation $\pi^{\tt l}$, we have
\[
\frac{|S_{\pi^{\tt l}}|}{2\beta_{\pi^{\tt l}}} \le \frac52\, \frac{|S_{\pi^{\tt l}}|}{B}\,,
\]
whence, by Lemma~\ref{lem:mainanticoncentration} applied at most
$ C|S_{\pi^{\tt l}}|/B+1$ times with $w_j=\pi^{\tt l}_j/\beta_{\pi^{\tt l}}$,
\[
\bP^{\xi, \pi^{\tt l}}\Bigl\{ \Bigl| \frac1{\beta_{\pi^{\tt l}}} \sum_{j=1}^m \pi^{\tt g}(j)
(\pi^{\tt l}\eta)'_j + m L_{\pi^{\tt l}}\Bigr|
\le \frac{|S_{\pi^{\tt l}}|}{2\beta_{\pi^{\tt l}}}\Bigr\}
\le \frac{C}m\, \Bigl( \frac{|S_{\pi^{\tt l}}|}{B} + 1 \Bigr)\,,
\]
where $ \bP^{\xi, \pi^{\tt l}} $ denotes expectation with respect to $\pi^{\tt g}$.

At last, recall that, given $\eta$, $S_{\pi^{\tt l}}=\sum_{j=1}^m (\pi^{\tt l}\eta)_j$ has the same distribution as $\sum_{j=1}^m \varepsilon_j \eta_j$
where $\varepsilon_j$ are i.i.d. Bernoulli random variables taking
	the values $\{-1,1\}$ with equal probability.
Then, denoting by $\bP^\eta$ the expectation with respect to $\eps_j$s, recalling that $B^2=\sum_{j=1}^m \eta_k^2$,
and using the subgaussian property of Bernoulli sums
(a.k.a. the Bernstein-Chernoff inequality), we have
\[
\bP^\xi \Bigl\{ \frac{|S_{\pi^{\tt l}}|}{B} \ge t \Bigr\} = \bP^\eta \Bigl\{ \Bigl| \sum_{j=1}^m \eps_k \eta_k \Bigr|
		\geq tB  \Bigr\} \le C e^{-c\,t^2}\,, \qquad t>0\,,
\]
whence
\begin{multline*}
\bP^{\xi}\Bigl\{ \Bigl| \frac1{\beta_{\pi^{\tt l}}} \sum_{j=1}^m \pi^{\tt g}(j)
(\pi^{\tt l}\eta)'_j + m L_{\pi^{\tt l}}\Bigr|
\le \frac{|S_{\pi^{\tt l}}|}{2\beta_{\pi^{\tt l}}}; \ \pi^{\tt l}\in\Slkgd \Bigr\} \\
\le \frac{C}m\Bigl(1+\sum_{n\ge 1} n \bP^\xi \Bigl\{ \frac{|S_{\pi^{\tt l}}|}{B} \ge n \Bigr\} \Bigr)
\le \frac{C}m\,.
\end{multline*}
Therefore, using Claim~\ref{lem-good},
\[
\bP^{\xi}\Bigl\{ \Bigl| \frac1{\beta_{\pi^{\tt l}}} \sum_{j=1}^m \pi^{\tt g}(j)
(\pi^{\tt l}\eta)'_j + m L_{\pi^{\tt l}}\Bigr| \le \frac{|S_{\pi^{\tt l}}|}{2\beta_{\pi^{\tt l}}}\Bigr\}
\le \frac{C}m +e^{-cm} \le \frac{C}m\,.
\]
This completes the proof of Lemma~\ref{lemma3} in the even case. \qed

\section{Anti-concentration for the symmetric group. Proof of Lemma~\ref{lem:mainanticoncentration}}

The proof of Lemma~\ref{lem:mainanticoncentration} goes in two steps; first, we prove the
anti-concentration estimate on the length-scale $\sqrt{n}$:

\medskip\noindent{\bf Lemma~$\bf 4'$.}
{\em Let $n\geq 2$ and let  $w_{1},\ldots ,w_{n}$ be real numbers such that $\sum_{i=1}^{n}w_{i}=0$ and $\sum_{i=1}^{n}w_{i}^{2}=1$. Let $\pi$ be a random
permutation
uniformly distributed on $\Sym_{n}$. Then, for every $L\in \bR$, we have
\begin{equation}\label{eq:ineqatrootnscale}
\bP\Bigl\{\Big| \sum_{i=1}^{n}w_{i}\pi (i) - Ln \Big| \le \sqrt{n}\Bigr\}\le\frac{C}{\sqrt{n}}\, e^{-c|L|}.
\end{equation}
}

\smallskip\noindent
Then we deduce from Lemma~$4'$ the full result, that is, the estimate
on the unit-length scale. Note that Lemma~$4'$ is weaker than Lemma~\ref{lem:mainanticoncentration},
since the latter Lemma implies the former one.

\subsection{Anti-concentration on the length-scale $\sqrt{n}$: proof of Lemma~$4'$}\label{subsec:actualproofofanticoncentration}

\subsubsection{Preliminaries}
Here,
we introduce notation and  random variables to be used in the proof of Lemma~$4'$.

\paragraph{Some notation.}
For each $\sigma\in \Sym_{n}$, let
$\Delta_{\sigma}=\{x\in [0,1]^{n}\colon x_{\sigma (1)}< x_{\sigma (2)}<
\ldots <x_{\sigma (n)}\}$. Then $\Delta_{\sigma}$ is a simplex.
The simplices $\Delta_{\sigma}$ and $\Delta_{\sigma'}$ are disjoint if
$\sigma\neq\sigma'$, and the union (over all $\sigma\in \Sym_{n}$)
of these simplices has full Lebesgue measure in the unit cube $[0,1]^{n}$. The barycenter of the simplex
$\Delta_{\sigma}$ is $P_{\sigma}:=\frac{1}{n+1}\sigma^{-1}$, where
$\sigma^{-1}=(\sigma^{-1}(1),\sigma^{-1}(2),\ldots ,\sigma^{-1}(n))$.

As above, we denote by $\sigma u$ the action of the permutation
$\sigma\in\Sym_n$ on the vector
$u=(u_1, \ldots , u_n)$, that is, $(\sigma u)_j = u_{\sigma (j)}$. We let
\[
u(\sigma)^2 = \sum_{j=1}^n \bigl( u_{\sigma (j)} + \ldots + u_{\sigma (n) }\bigr)^2\,.
\]

\paragraph{Random variables.}
Next we introduce certain important random variables. Throughout,
$V$ denotes a random variable with  uniform distribution in the unit
cube $[0,1]^{n}$. In coordinates, $V_{1},\ldots ,V_{n}$ are i.i.d. random variables
with uniform distribution on $[0,1]$. For given $\sigma\in \Sym_n$,
$V_{\sigma}$ denotes a random variable
with uniform  distribution in the simplex $\Delta_{\sigma}$.
Then,
$(V_{\sigma})_{\sigma (1)},\ldots ,(V_{\sigma})_{\sigma (n)}$
have the same joint distribution as
the order statistics of $n$ independent uniform random variables on $[0,1]$.

\paragraph{A bound on the density of a sum of independent uniform random variables.}
In the proof of Lemma~$4'$, we will use the following lemma:
\begin{lemma}\label{lem:unimodal} Let $w_{i}$ be real
	{\rm (}non-random{\rm )}
	numbers and let $U_{i}$ be  i.i.d. random variables with uniform distribution on
    $\bigl[ -\frac12, \frac12 \bigr]$. Let $X=\sum_{i=1}^{n}w_{i}U_{i}$. Let
    $p_{X}(\cdot)$ be the density of $X$.
\begin{enumerate}
\item If $\sum_{i=1}^{n}w_{i}^{2}=1$, then $p_{X}(t)\le Ce^{-c |t|}$ for all $t\in \bR$.
\item If $\sum_{i=1}^{n}w_{i}^{2}\le 1$, then $p_{X}(t)\le Ce^{-c |t|}$ for $|t|>1$.
\end{enumerate}
\end{lemma}
Lemma \ref{lem:unimodal} is probably known but for completeness
we prove it in Section~\ref{sec:proofofprobestimates}.
We also indicate therein how to get the stronger bound
$p_{X}(t)\le Ce^{-c t^{2}}$.
It will become clear that if we used that improved upper bound in place of
Lemma \ref{lem:unimodal},
we would get the bound $\frac{C}{\sqrt{n}}e^{-cL^{2}}$ and $\frac{C}{n}e^{-cL^{2}}$
in Lemmas~$4'$ and~\ref{lem:mainanticoncentration}, respectively.

\subsubsection{Idea of the proof of Lemma~$4'$} Our goal is to prove estimate~\eqref{eq:ineqatrootnscale}.
Two difficulties are (a) discreteness of the random variable $\pi$, and (b) dependence between the
random variables  $\pi(1),\ldots ,\pi(n)$. This motivates considering first the following ``baby-version''
of Lemma~\ref{lem:mainanticoncentration} that does not have these difficulties and is a straightforward
consequence of the bound for the density of $\langle w,V \rangle$ as given in Lemma~\ref{lem:unimodal}.

\paragraph{A baby-version of Lemma~$4'$.}
{\em Let $V_{i}$ be i.i.d. random variables having uniform distribution on $[0,1]$. Let $w_{1},\ldots ,w_{n}$
be real numbers satisfying $\sum_{i=1}^{n}w_{i}=0$ and $\sum_{i=1}^{n}w_{i}^{2}=1$. Then, for any $L\in \bR$
and $t>0$, we have
\[
\bP \Bigl\{ \Big|\sum_{i=1}^{n}w_{i}V_{i}-L \Big|\le t \Bigr\}\le Cte^{-c(|L|-t)_{+}}.
\]
In particular, for $t=\frac{1}{\sqrt{n}}$ or $t=\frac{1}{n}$ we get the bounds $\frac{C}{\sqrt{n}}e^{-c|L|}$ and $\frac{C}{n}e^{-c|L|}$, respectively.
}

\medskip
When we scale $V_{i}$ up by $n$ (so that they are uniform on $[0,n]$) then similarity with Lemma~$4'$ becomes clear.
Apart from analogy, observe that if $\pi$ is uniformly distributed on $\Sym_n$, then $\pi(i)$ (for any $i$) is uniformly distributed in $\{1,2,\ldots ,n\}$, and any
finite number of them, $\pi(i_{1}),\ldots ,\pi(i_{k})$, are nearly independent (for large $n$).

\paragraph{Bad permutations.}
We need to count {\em bad} permutations $\pi$ such that
 \[
 \Bigl| \sum_{i=1}^{n}w_{i} \pi^{-1}(i)-Ln \Bigr|\le \sqrt{n}\,.
 \]
This is almost the same as  $|\<w,P_{\pi}\>-L|\le \frac{1}{\sqrt{n}}$
(not exactly the same because $P_{\pi} = \frac{\pi^{-1}}{n+1}$,
which is not exactly $ \frac{\pi^{-1}}n$, but that difference will be shown to be harmless). Let
$\BB$ denote the set of bad permutations.

Let $f\colon [0,1]^{n}\to \bR_{+}$ be a measurable function. Then
\begin{equation}
\frac1{n!}\, | \BB | \cdot \min\limits_{\pi\in\BB}\, \bE[f(V_{\pi})] \le \bE[f(V)]. \label{eq:modifiedmarkov}
\end{equation}
 The idea is to find a function $f$ for which we can find an upper bound for $\bE[f(V)]$ and a lower bound for $\bE[f(V_{\pi})]$ for any $\pi\in \BB$.

\paragraph{A choice of the function $f$.}
The first natural choice would be $f(x)=\done_{\{|\<w,x\>-L|\le s\}}$ for an appropriate value of $s$.
The reason it fails is that although $\bE[V_{\pi}]=P_{\pi}$, there are many $\pi\in \BB$ for which the variance of $\<w,V_{\pi}\>$ (which is at most
$w(\pi)^{2}/n^{2}$ by a simple estimate given in Claim~\ref{lem:varianceofwdotvpi} below) is quite large.

We enhance the previous choice by taking $f(x)=F(x)\done_{\{|\<w,x\>-L|\le s\}}$ for an appropriately chosen $s$,
and with the choice
\begin{align}\label{eq:Fasalinfunctiononsimplex}
	F(x)=\sum_{j=1}^{n}(w_{\pi (j)} + \ldots + w_{\pi (n)})^{2}(x_{\pi (j)}-x_{\pi (j-1)}) \;\;\;\quad {\rm{ if\ }} x\in \Delta_{\pi}.
\end{align}
The key point here is that $\bE[F(V_{\pi})]=w(\pi)^{2}/(n+1)$, which is large precisely for the troublesome permutations (those for which
$\operatorname{Var}[\<w,V_{\pi}\>]\le w(\pi)^{2}/n^{2}$ is large). Thus, we can get a better lower bound for $\bE[f(V_{\pi})]$ for $\pi\in \BB$.
It turns out that we still retain a good upper bound for  $\bE[f(V)]$.

\medskip
Note that, in the proof, $\Sym_n$ will be broken into disjoint groups based on the value of $w(\pi)$, and inequality~\eqref{eq:modifiedmarkov}
will be applied within each group and then summed over the groups.

\subsubsection{Beginning of the proof of Lemma~$4'$}
Recall that $w(\pi)^{2}=\sum_{j=1}^{n}(w_{\pi (j)} + \ldots + w_{\pi (n)})^{2}$.
Using Cauchy-Schwarz and the normalization $\sum_{j=1}^{n}w_{j}^{2}=1$, we see that $w (\pi)^{2}\le n^{2}$ for all $\pi\in \Sym_{n}$.
Let $Q\ge 10$ be a fixed constant (its value is unchanged throughout this section).
We define the following sets whose union is all of $\Sym_{n}$.
\begin{itemize}
\item $\Sym (0):= \{\pi\in\Sym_n\colon w(\pi)\le 4(|L|+Q)\sqrt{n}\}$ and
\item $\Sym (\ell):=\{\pi\in\Sym_n\colon 2^{\ell-1}< w(\pi)\le 2^{\ell}\}$ for $\ell$ such that $4(|L|+Q)\sqrt{n}\le 2^{\ell}\le 2n$.
\end{itemize}
The goal is to prove the inequality~\eqref{eq:ineqatrootnscale}. We claim that it follows if we prove that
\begin{align}\label{eq:inequatrootnscale3}
\bP\bigl\{ |\<w,P_{\pi}\>-L| \le \frac{1}{\sqrt{n}}\bigr\}\le \frac{C}{\sqrt{n}}e^{-c|L|}.
\end{align}
Indeed, what we want to bound in~\eqref{eq:ineqatrootnscale} is
\begin{align*}
\bP\Bigl\{ \Big| \sum_{i=1}^{n}w_{i}\pi (i) - Ln \Big| \le \sqrt{n}\Bigr\} &= \bP\Bigl\{ \Big| \frac{1}{n+1} \sum_{i=1}^{n}w_{i}\pi (i) - L\frac{n}{n+1} \Big| \le
\frac{\sqrt{n}}{n+1}\Bigr\} \\
&\le \bP \Bigl\{ \Big|\<w,P_{\pi}\> - L \frac{n}{n+1}\Big|\le \frac{1}{\sqrt{n}}\Bigr\}
\end{align*}
because $P_{\pi}=\frac{1}{n+1}\pi^{-1}$ and $\pi^{-1}$ has the same distribution as $\pi$.
Applying~\eqref{eq:inequatrootnscale3} we get
\[
\bP \Bigl\{ \Big| \sum_{i=1}^{n}w_{i}\pi (i) - Ln \Big| \le \sqrt{n}\Bigr\} \le \frac{C}{\sqrt{n}}e^{-c|Ln/(n+1)|}
\le \frac{C}{\sqrt{n}}e^{-\frac{c}{2}|L|}.
\]
Let $\BB=\{\pi\in\Sym_n\colon |\<w,P_{\pi}\>-L|\le \frac{1}{\sqrt{n}}\}$ be the set of all ``bad'' permutations.
We shall get bounds for the cardinality of $\BB\cap \Sym (\ell)$ and thus get a bound for $\bP\{\BB\}$.

Before starting the proof, we note that it suffices to prove~\eqref{eq:inequatrootnscale3} for $n\ge n_{0}$ for some fixed $n_{0}$.
The reason is that for $n\le n_{0}$,
\[
|\<w,P_{\pi}\>| \le \sqrt{\sum_{i=1}^{n}w_{i}^{2}} \cdot \sqrt{\frac{1}{(n+1)^{2}}\sum_{i=1}^{n}(\pi^{-1}(i))^{2}} \le L_{0}
\]
for a constant $L_{0}$ (not depending on the choice of $w_{i}$s or $\pi$ or $n$). Hence by choosing $C$ so large that $\frac{C}{\sqrt{n_{0}}}e^{-c(L_{0}+1)}\ge 1$,
the inequality~\eqref{eq:inequatrootnscale3} is trivially satisfied  for all $n\le n_{0}$ and for all $L\in \bR$
(for $|L|>L_{0}+1$, the probability is zero while, for $|L|\le L_{0}+1$, the right-hand side in~\eqref{eq:inequatrootnscale3} is bigger than $1$).

\subsubsection{A bound on $\operatorname{Var}[ \langle w, V_\sigma\rangle ]$}

To proceed, we need to bound the variance of $\langle w, V_\sigma\rangle$.
\begin{claim}\label{lem:varianceofwdotvpi} Let $w=(w_{1},\ldots , w_{n})$ be a vector in $\bR^{n}$. Let
$\sigma\in\Sym_n$, and let
$V_\sigma$ be uniform on $\Delta_{\sigma}$. Then,
\[
\operatorname{Var}[ \langle w,V_{\sigma} \,\rangle ]\le \frac{w(\sigma)^2}{(n+1)(n+2)}\,.
\]
\end{claim}
The proof of this claim is rather straightforward. We will give it in Section~\ref{sect:proof-of-claim9}.

\subsubsection{A bound on the cardinality of $\Sym_n^{\tt bad} \cap \Sym (0)$}
Let $f: [0,1]^{n}\to \bR_{+}$ be defined by
\[
f(x)=\done_{\{ |\<w,x\>-L|\le \frac{9(|L|+Q)}{\sqrt{n}}\}}\,.
\]
If $\pi\in \BB\cap \Sym (0)$, then $w(\pi)\le 4\sqrt{n}(|L|+Q)$ and $|\<w,P_{\pi}\>-L|\le \frac{1}{\sqrt{n}}$. Note that $\<w,V_{\pi}\>$ has
mean $\<w,P_{\pi}\>$ and variance at most $\frac{1}{n^{2}} w(\pi)^{2}$ (by Claim~\ref{lem:varianceofwdotvpi}). Therefore, by Chebyshev's inequality,
\[
\Big| \<w,V_{\pi}\>-\<w,P_{\pi}\> \Big| \le 2\frac{w(\pi)}{n} \;\;\;\; \quad \text{  with probability at least }\frac12\,.
\]
By the bound on $w(\pi)$, we see that
\[
|\<w,V_{\pi}\>-L|\le \frac{1+8(|L|+Q)}{\sqrt{n}} \;\;\;\; \quad \text{  with probability at least }\frac12\,.
\]
As $Q\ge 10$, we can write $1+8(|L|+Q)\le 9(|L|+Q)$ and hence,
\begin{equation}\label{eq:lbdforS0}
\bE[f(V_{\pi})]\ge \frac12 \;\;\; \quad \text{ for } \pi\in \BB\cap \Sym (0).
\end{equation}
Now we find an upper bound for
\[
\bE[f(V)] = \bP \bigl\{ | \<w,V\>-L|\le \frac{9(|L|+Q)}{\sqrt{n}} \bigr\}.
\]
Write $\<w,V\>=\sum_{i=1}^{n}w_{i}V_{i}=\sum_{i=1}^{n}w_{i}(V_{i}-\frac12 )$ (since $\sum_{i=1}^{n}w_{i}=0$) and recall that
$\sum_{i=1}^{n}w_{i}^{2}=1$ to  see that the first part of  Lemma~\ref{lem:unimodal} is applicable. It  gives
\begin{align}\label{eq:lemunimodalfirst}
\bP \bigl\{| \<w,V\>-L|\le \frac{9}{\sqrt{n}}(|L|+Q) \bigr\} \le C\frac{|L|+Q}{\sqrt{n}}e^{-c(|L|-\frac{9}{\sqrt{n}}(|L|+Q) )_{+} }.
\end{align}
If $|L|<1$, we drop the exponential term, and multiply by $e^{-|L|+1}$ which is at least $1$. If $|L|\ge 1$, then for $n\ge n_{0}$, we have
$|L|-\frac{9(|L|+Q)}{\sqrt{n}}\ge \frac12 |L|$. Thus, in either case, we get the bound (for $n\ge n_{0}$)
\begin{align}\label{eq:ubdforS0}
\bE[f(V)] \le \frac{C}{\sqrt{n}}(|L|+Q)e^{-\frac{c}{2} |L|} \le \frac{C}{\sqrt{n}}e^{-\frac{c}{4}|L|}
\end{align}
since $x\to (x+Q)e^{-cx/4}$ is bounded for $x\in (0,\infty)$.

Invoking~\eqref{eq:modifiedmarkov}, we conclude from~\eqref{eq:lbdforS0} and~\eqref{eq:ubdforS0} that, for $n\ge n_{0}$ and for all $L\in \bR$,
\begin{align}\label{eq:finalbdforS0}
\frac{1}{n!}|\BB\cap \Sym (0)|\le \frac{C}{\sqrt{n}}e^{-c|L|}.
\end{align}

\subsubsection{A bound on the cardinality of $\Sym_n^{\tt bad} \cap \Sym (\ell)$ with $4(|L|+Q)\sqrt{n}\le 2^\ell \le 2n$}

Fix $T=2^{\ell}$ so that $\frac{T}{2}\le w(\pi)\le T$. Let $f_T\colon [0,1]^{n}\to \bR_{+}$ be defined as
\[
f_T(x)=F(x)\done_{ \{ |\<w,x\>-L|\le \frac{2QT}{n} \} }\,,
\]
where  $F$ is the function that was defined in \eqref{eq:Fasalinfunctiononsimplex}.

\paragraph{A lower bound for $\bE[f(V_\pi)]$, $\pi\in\Sym_n^{\tt bad} \cap \Sym (\ell)$.}
The random variable $\<w,V_{\pi}\>$ has mean $\<w,P_{\pi}\>$ and Claim~\ref{lem:varianceofwdotvpi}  asserts that
$\operatorname{Var}[\<w,V_{\pi}\>] \le \frac{T^{2}}{n^{2}}$. By Chebyshev's inequality,
\[
|\<w,V_{\pi}\>-\<w,P_{\pi}\>|\le Q\frac{T}{n} \;\;\;\;\quad \text{  with probability at least } 1-\frac{1}{Q^{2}}\,.
\]
If $\pi\in \BB\cap \Sym(\ell)$, then, in addition to the above, we have $|\<w,P_{\pi}\>-L|\le \frac{1}{\sqrt{n}}$. Therefore, recalling that $Q\ge 10$,
\begin{align}\label{eq:lbdforf-1}
|\<w,V_{\pi}\>-L|\le \frac{\sqrt{n}+QT}{n} \;\;\;\;\quad \text{  with probability at least } 0.99\,.
\end{align}

By the definition~\eqref{eq:Fasalinfunctiononsimplex}, we can write
\[
F(V_{\pi}) = \sum_{j=1}^{n}\alpha_{j}((V_{\pi})_{\pi (j)}-(V_{\pi})_{\pi (j-1)}) \qquad \text{ with }
\alpha_{j}=(w_{\pi (j)}+\ldots + w_{\pi (n)})^{2}\,.
\]
Then, recall that $\bE[V_{\pi}]=P_{\pi}$ which is $\frac{1}{n+1}\pi^{-1}$. Hence, $\bE[(V_{\pi})_{\pi (j)}]=\frac{j}{n+1}$ and
\[
\bE[F(V_{\pi})] = \frac{1}{n+1}\sum_{j=1}^{n}\alpha_{j} = \frac{1}{n+1}\, w(\pi)^{2}.
\]

We may also rewrite $F(V_{\pi})$ as $\sum_{j=1}^{n}(\alpha_{j}-\alpha_{j+1})(V_{\pi})_{\pi (j)}$ (with the convention that $\alpha_{n+1}=0$). Then, applying
Claim~\ref{lem:varianceofwdotvpi} we get $\operatorname{Var}[F(V_{\pi})] \le \frac{1}{(n+1)^{2}}\, \sum_{j=1}^{n}\alpha_{j}^{2}$. By the second moment
inequality\footnote{
The second moment inequality asserts that if $X$ is a non-negative random variable, then $\bP\{X\ge \lambda \bE[X]\}\ge
(1-\lambda)^{2}\frac{(\bE[X])^{2}}{\bE[X^{2}]}$ for  $0<\lambda<1$. The proof is a straightforward application of the Cauchy-Schwarz inequality,
see~\cite[p.8]{kahane}.}
\begin{align*}
\bP \bigl\{ F(V_{\pi}) \ge \frac12 \bE[F(V_{\pi})] \bigr\} &\ge
\frac{1}{4}\, \frac{(\bE[F(V_{\pi})])^{2}}{\operatorname{Var}[F(V_{\pi})]+(\bE[F(V_{\pi})])^{2}}  \\
&= \frac{1}{4}\, \Bigl( 1 + \frac{\operatorname{Var}[F(V_{\pi})]}{(\bE[F(V_{\pi})])^{2}}\Bigr)^{-1} \\
&\ge \frac{1}{4}\, \Bigl(1+\frac{\sum_{j=1}^{n}\alpha_{j}^{2}}{(\sum_{j=1}^{n}\alpha_{j})^{2}} \Bigr)^{-1}\,.
\end{align*}
Since
$\alpha_{j}$s are all non-negative, $(\sum_{j}\alpha_{j})^{2}\ge \sum_{j}\alpha_{j}^{2}$ and hence
\[
\bP \bigl\{ F(V_{\pi})\ge \frac12 \bE[F(V_{\pi})] \bigr\} \ge \frac{1}{8}\,.
\]
Combine this with~\eqref{eq:lbdforf-1} to conclude that
\[
\bP \bigl\{F(V_{\pi})\ge \frac{1}{2(n+1)} w(\pi)^{2}  \ \text{ and }\ |\<w,V_{\pi}\>-L|\le \frac{2QT}{n}\bigr\} \ge \frac{1}{10}.
\]
Consequently,
\begin{align}\label{eq:lbdforSl}
\bE[f_T(V_{\pi})]\ge \frac{w(\pi)^2}{2(n+1)}\cdot\frac1{10}\, \stackrel{w(\pi)\ge T/2}\ge \,
\frac{T^2}{80(n+1)} \ge \frac{T^{2}}{160 n} \;\;\;\quad \text{ for }\pi\in \BB\cap \Sym (\ell)\,.
\end{align}

\paragraph{An upper bound for $\bE[f_T(V)]$.}
For this, we write $F(x)$ in the following alternative form.
\[
F(x)=\int_{0}^{1}(G_{t}(x))^{2}{\rm d}t \;\;\;\quad \text{ where } G_{t}(x)=\sum_{j=1}^{n}w_{j}\done_{\{x_{j}>t\}}\,.
\]
It is easy to check that this agrees with \eqref{eq:Fasalinfunctiononsimplex}. Hence,
\begin{align}\label{eq:newformforf}
\bE[f_T(V)] = \int_{0}^{1} \bE \bigl[ G_{t}(V)^{2}\done_{\{|\<w,V\>-L|\le \frac{2QT}{n}\}} \bigr] {\rm d}t\,.
\end{align}
Fix $t\in (0,1)$ and let  $S_{t}=\{i\colon V_{i}>t\}$. We estimate the integrand for each $t$.

\medskip
Write
\begin{align}\label{eq:ubdforSl1-1}
\bE \Bigl[ G_{t}(V)^{2}\done_{\{|\<w,V\>-L|\le \frac{2QT}{n}\}} \Bigr]  &\le 64(|L|+2Q)^{2}\bP \Bigl\{ |\<w,V\>-L|\le \frac{2QT}{n} \Bigr\} \nonumber \\
&\quad  + \bE\Bigl[ G_{t}(V)^{2}\done_{\{|G_{t}(V)|\ge 8(|L|+2Q)\}}\, \done_{\{|\<w,V\>-L|\le \frac{2QT}{n}\}}\Bigr]\,.
\end{align}
We bound the first term along similar lines to the case $\ell=0$. Since $\sum_{i=1}^{n}w_{i}^{2}=1$,
we may apply the first part of Lemma~\ref{lem:unimodal} to the random variable $\<w,V\>$ to get
\begin{align}\label{eq:lemunimodalsecond}
\bP \Bigl\{ |\<w,V\>-L|\le \frac{2QT}{n} \Bigr\} \le C\,\frac{2QT}{n}\,\exp\Bigl\{ -c\Bigl(|L|-\frac{2QT}{n}\Bigr)_{+} \Bigr\}\,.
\end{align}
Since $T\le 2n$ we see that $\frac{2QT}{n}\le 4Q$, a constant. Hence, dividing into the cases $|L|\ge 8Q$ and $|L|\le 8Q$, and changing constants suitably, for $n\ge
n_{0}$ and for all $L\in \bR$ we have the inequality
\[
\bP \Bigl\{ |\<w,V\>-L|\le \frac{2QT}{n}\Bigr\} \le C\,\frac{T}{n}\,\exp\bigl\{-c|L|\bigr\}\,.
\]
Thus, the first term in~\eqref{eq:ubdforSl1-1} is bounded by (for $n\ge n_{0}$ and for all $L\in \bR$)
\begin{equation}\label{eq:ubdforSl1-2}
64(|L|+2Q)^{2} \bP \bigl\{ |\<w,V\>-L|\le \frac{2QT}{n} \bigr\} \le  C\,\frac{T}{n}\,(|L|+2Q)^{2}e^{-c|L|}
\le C\,\frac{T}{n}\,e^{-c|L|/2}\,,
\end{equation}
again because $x\mapsto (x+2Q)^{2}e^{-cx/2}$ is bounded.

\medskip
It remains to control the second term in \eqref{eq:ubdforSl1-1}.
The trick is to condition on the random set $S_{t}=\{i\colon V_{i}>t\}$. We need to understand the conditional distributions of
the two random variables $G_{t}(V)$ and $\<w,V\>$. The first one is easy since $G_{t}(V)=\sum_{i\in S_{t}}w_{i}$,
which is a function of $S_{t}$. In other words, conditional on $S_{t}$, the variable $G_{t}(V)$ is a constant.

Next, conditional on the set $S_{t}$, the random variables $V_{1},\ldots ,V_{n}$ are still independent, $V_{i}$ is uniformly distributed on $[t,1]$ if $i\in S_{t}$
and $V_{i}$ is uniformly distributed in $[0,t]$ if $i\not\in S_{t}$. Let $U_{i}$ be independent random variables distributed uniformly on $[ -\frac12 ,\frac12 ]$ and
set
\[
V_{i}'=\begin{cases}
\frac12 t+tU_{i} & \text{ if }i\not\in S_{t}, \\
\frac12 (1+t)+(1-t)U_{i} & \text{ if }i\in S_{t}.
\end{cases}
\]
Then, the distribution of the vector $V'=(V_{1}',\ldots ,V_{n}')$ is the same as the conditional distribution of $V$ given $S_{t}$. In particular, the conditional
distribution of $\<w,V\>$ given $S_{t}$ is the same as the unconditional distribution of
\begin{align*}
\<w,V'\> &= \frac12\, (1+t) \sum_{i\in S_{t}}w_i + \sum_{i\in S_{t}}w_{i}(1-t)U_{i} + \frac12\, t  \sum_{i\not\in S_{t}}w_{i} + \sum_{i\not\in
S_{t}}w_{i}tU_{i} \\
&= \frac12\, G_{t}(V) + \sum_{i\in S_{t}}w_{i}(1-t)U_{i}  +\sum_{i\not\in S_{t}}w_{i}tU_{i}.
\end{align*}
The quantity we need to control is
$\bE\Bigl[ G_{t}(V)^{2}\done_{\{ G_{t}(V)\ge 8(|L|+2Q)\}} \done_{\{ |\<w,V\>-L|\le \frac{2QT}{n} \}}\Bigr]$.
By conditioning on
$S_{t}$, we may write this as
\begin{align}\label{eq:anamika1}
\bE\Bigl[ G_{t}(V)^{2}\done_{\{|G_{t}(V)|\ge 8(|L|+2Q)\}}\, \bE\Bigl[ \done_{\{|\<w,V\>-L|\le \frac{2QT}{n}\}} \;\Big| \; S_{t}\Bigr]\, \Bigr]\,.
\end{align}
(As $G_{t}(V)$ is a function of $S_{t}$, factors involving it can be
taken out of the conditional expectation.) Using our representation
of the conditional distribution of $\<w,V\>$ in terms of
the $U_{i}$s, the inner conditional expectation may be written as
\begin{align*}
\bE\Bigl[ \done_{\{|\<w,V\>-L|\le \frac{2QT}{n}\}} \;\Big| \; S_{t}\Bigr]
&= \bE\Bigl[\done_{\{|\<w',V\>-L|\le \frac{2QT}{n}\}}\;\Big|\; S_t\Bigr] \\
&= \bP\Bigl\{\Big|\sum_{i=1}^{n}w_{i}'U_{i}+\frac12 G_{t}(V)-L \Big|
\le \frac{2QT}{n} \;\Big|S_t\;\Bigr\}\,,
\end{align*}
where $w_{i}'=(1-t)w_{i}$ if $i\in S_{t}$ and $w_{i}'=tw_{i}$ if $i\not\in S_{t}$.
Again, we want to apply the density
bound from Lemma~\ref{lem:unimodal}. However, we only have the upper bound
$\sum_{i=1}^{n}(w_{i}')^{2}\le 1$, and hence, to apply
the second part of that lemma, we need to make sure that
at least on the event $\{|G_t(V)|\geq 8(|L|+2Q)\}$,  see~\eqref{eq:anamika1},
the interval
\[
\Bigl[ \frac12 G_{t}(V)-L-\frac{2QT}{n}, \frac12 G_{t}(V)-L+\frac{2QT}{n} \Bigr]
 \]
 is at distance at least $1$ from the origin. Once we show that, the second part of Lemma~\ref{lem:unimodal} gives the bound
\begin{align}\label{eq:anamika2}
\bE\Bigl[ \done_{\{|\<w,V\>-L|\le \frac{2QT}{n}\}} \;\Big| \; S_{t}\Bigr] \le
C\,\frac{2QT}{n}\,\exp\Bigl\{-c\Bigl(\frac12 |G_{t}(V)|-|L|-\frac{2QT}{n} \Bigr)_{+} \Bigr\}\,.
\end{align}
But
on the event $\{|G_t(V)|\geq 8(|L|+2Q)\}$,
recalling  that $\frac{2QT}{n}\le 4Q$, one has
\[
\Big| \frac12 G_{t}(V)-L\pm\frac{2QT}{n} \Big| \ge \frac{1}{2}|G_{t}(V)|- |L|-4Q \ge 3|L|+4Q\,,
\]
which is at least $1$. Thus, the bound~\eqref{eq:anamika2} is valid.

We now multiply the left hand side of~\eqref{eq:anamika2} by $G_{t}(V)^{2}\done_{\{G_{t}(V)\ge 8(|L|+2Q)\}}$ and take expectations.
If  the indicator is to be non-zero, then
\[
\frac12\, G_{t}(V)-|L|-\frac{2QT}{n} \ge \frac{1}{8}\, G_{t}(V) + 3(|L|+2Q)-|L|-4Q = \frac{1}{8}\, G_{t}(V) + 2|L|+2Q.
\]
Using this bound along with~\eqref{eq:anamika2}, we get
\[
\text{the second term in }\eqref{eq:ubdforSl1-1} \le C\frac{T}{n}\, \bE \Bigl[ G_{t}(V)^{2}\exp\Bigl\{-c\bigl(\frac{1}{8}\,G_{t}(V) + 2|L|+2Q\bigr) \Bigr\} \Bigr]
\le C\,\frac{T}{n}\,e^{-c|L|}
\]
since $G_{t}(V)^{2}\exp\{-cG_{t}(V)\}$ is bounded by a constant.

Adding this to the bound for the first term given in~\eqref{eq:ubdforSl1-1}, we arrive at
\[
\bE\bigl[ G_{t}(V)^{2}\done_{\{|\<w,V\>-L|\le \frac{2QT}{n}\}} \bigr] \le C\,\frac{T}{\sqrt{n}}\,e^{-c|L|}.
\]
Integrating this bound over $t$ and plugging into~\eqref{eq:newformforf} gives us
\begin{align}\label{eq:ubdforSl}
\bE[f_T(V)] \le C\,\frac{T}{n}\,e^{-c|L|}\,.
\end{align}

\paragraph{Tying the ends together.}
At last, juxtaposing~\eqref{eq:lbdforSl} and~\eqref{eq:ubdforSl}, and recalling that $ T=2^\ell $,
we conclude that
\begin{align}\label{eq:finalbdforSl}
\frac{1}{n!}|\BB\cap \Sym (\ell)| \le \frac{\bE[f_T(V)]}{\min\limits_{\pi\in\BB} \bE[f_T(V_\pi)]} \le
\frac{C\,n}{T^2} \cdot \frac{C\,T}{n}\,e^{-c|L|}
= \frac{C}{T}\,e^{-c|L|}\,
=\, \frac{C}{2^{\ell}}\,e^{-c|L|}\,.
\end{align}

\subsubsection{End of the proof of Lemma~$4'$}

Adding up~\eqref{eq:finalbdforSl} over $\ell$ such that $2^{\ell}\ge 4(|L|+Q)\sqrt{n}$
and recalling the estimate for $\frac{1}{n!}|\BB\cap \Sym (0)|$ we got in~\eqref{eq:finalbdforS0},
we obtain
\begin{align*}
\frac1{n!} | \BB | &\le \frac1{n!} | \BB \cap \Sym (0)|\ +
\sum_{2^{\ell}\ge 4(|L|+Q)\sqrt{n}}\, \frac1{n!} | \BB \cap \Sym (\ell)| \\
&\le
\frac{C}{\sqrt{n}}\, e^{-c|L|}\ + \sum_{2^{\ell}\ge 4(|L|+Q)\sqrt{n}}\, \frac{C}{2^{\ell}}\,e^{-c|L|}
\le \frac{C}{\sqrt{n}}\, e^{-c|L|}\,.
\end{align*}
This completes the proof of Lemma~$4'$. \hfill $\Box$

\subsection{From the scale $\sqrt{n}$ to the unit length-scale}\label{sec:inequatunitscale}

For the permutation $\pi\in \Sym_{n}$, define its ``weight'' as
$\text{wt}[\pi]=\sum_{j=1}^{n} \pi (j) w_j=\sum_{j=1}^{n}jw_{\pi (j)}$.  Fix $L$ and define the set of ``bad permutations'' $\BB$ as
the set of all $\pi\in \Sym_{n}$ for which  $| \text{wt}[\pi]-Ln|\le 1$.

\subsubsection{Partition of $\Sym_n$}

For each permutation $\sigma =(\sigma (1),\ldots ,\sigma (n-1))\in \Sym_{n-1}$, let
\begin{align*}
\widehat{\Sym}_{\sigma} &=
\bigl\{ \pi^{(1)}_{\sigma} = (n,\sigma (1),\ldots ,\sigma (n-1)\,), \,
\pi^{(2)}_{\sigma} = (\sigma (1), n, \sigma (2), \ldots , \sigma (n-1)),\, \ldots \\
&\quad \ldots\,, \,\pi^{(n)}_{\sigma}=(\sigma (1),\ldots ,\sigma (n-1), n)\bigr\}.
\end{align*}
Then, $\{\widehat{\Sym}_{\sigma}\colon \sigma\in \Sym_{n-1}\}$ is a partition of $\Sym_{n}$ into groups of $n$ permutations each.
The key point is to show that, for most $\sigma$, there are only a few bad permutations in $\widehat{\Sym}_{\sigma}$.

\subsubsection{The subset $\cA\subset \Sym_{n-1}$ }

Let $b=\max_{j}w_{j}-\min_{j}w_{j}$ and note that, due to normalization $\sum w_j = 0$ and $\sum w_j^2 =1$,
we have  $\frac{1}{\sqrt{n}}\le b\le 2$. We will need two subsets, $\cA$ and $\cE$ of $\Sym_{n-1}$.  These sets are ``exceptional'',
in the sense that they have small cardinality. Here, we will define $\cA$, the set $\cE$ will be defined later.

Let $\cA$ be the set of $\sigma\in \Sym_{n-1}$ for which $\text{wt}[\pi]\in [Ln-1-nb,Ln+1+nb]$ for all $\pi\in \widehat{\Sym}_{\sigma}$.
Then,
\begin{align*}
\frac{1}{(n-1)!}|\cA| &\le \frac{1}{n!}\Big|\{\pi\in \Sym_{n}\colon \text{wt}[\pi]\in [Ln-1-nb,Ln+1+nb]\}\Big| \\
&= \bP \bigl\{ \text{wt}[\pi]\in [Ln-1-nb,Ln+1+nb] \bigr\}\,.
\end{align*}
Divide the interval $[Ln-1-nb,Ln+1+nb]$ into $\lceil\frac{nb+1}{\sqrt{n}} \rceil$ intervals of length $2\sqrt{n}$ (or less) each, and
invoke Lemma~$4'$ for each of them. All the intervals are at distance at least $(|L|-b-\frac{1}{n})n$ from the origin, and hence
\begin{align}
\frac{1}{(n-1)!}|\cA| \le \Big\lceil\frac{nb+1}{\sqrt{n}} \Big\rceil\, \frac{C}{\sqrt{n}}\,e^{-c(|L|-3)_{+}}
\le Cbe^{-c|L|}\,.  \label{eq:firstcrudebound}
\end{align}

\subsubsection{A relabelling of $w_1$, ...., $w_n$}
To define the exceptional set $\cE$ we consider three cases: $b\ge \tfrac1{12}$, $n^{-\frac13}\le b \le \tfrac1{12}$,
and $n^{-\frac12}\le b \le n^{-\frac13}$.
First, we arrange the $w_{i}$s in a suitable manner:
\begin{enumerate}
\item $b\ge \frac{1}{12}$. Label $w_{i}$s so that $w_{n}=\max_{j}w_{j}$ or $w_{n}=\min_{j}w_{j}$, whichever of the two is larger in absolute value.
\item $n^{-\frac{1}{3}}\le b\le \frac{1}{12}$. As in the first case, we label $w_{i}$s so that $w_{n}=\max_{j}w_{j}$ or $w_{n}=\min_{j}w_{j}$, whichever of the two is larger in
    absolute
    value.
\item $\frac{1}{\sqrt{n}}\le b\le n^{-\frac{1}{3}}$. Label $w_{i}$s so that $w_{n}=\max_{j}w_{j}$ or $w_{n}=\min_{j}w_{j}$,

whichever of the two is such that the cardinality of $\{j\colon |w_{n}-w_{j}|\ge \frac{b}{2}\}$ is at least $\frac{n}{2}$.
\end{enumerate}
For simplicity of language, let us assume that $w_{n}=\max_{j}w_{j}$ in all cases
(in fact, there is no loss of generality as we may  negate all the $w_{j}$s and $L$ if needed).
Then, for any $\sigma\in \Sym_{n-1}$ and any $k\le n-1$, we have
\begin{align}\label{eq:weightsincreaseineachgroup}
\text{wt}[\pi^{(k+1)}_{\sigma}]-\text{wt}[\pi^{(k)}_{\sigma}] = w_{n}-w_{\sigma (k)}.
\end{align}
As $w_{n}=\max_{j}w_{j}$, these increments are non-negative. Thus, for any $\sigma$,
\begin{itemize}
\item we have $\text{wt}[\pi_\sigma^{(n)}] - \text{wt}[\pi_\sigma^{(1)}]\le nb$, and therefore,
if $\sigma\notin\cA$, then $\BB\cap \widehat{\Sym}_\sigma$ is empty;
\item
the set of $k$ for which $\pi^{(k)}_\sigma$ is bad, is a
discrete interval of the form $\{k_{\sigma}, k_{\sigma}+1, \ldots , k_{\sigma}+\ell_{\sigma}-1\}$.
\end{itemize}

\subsubsection{The exceptional set $\cE\subset\Sym_{n-1}$}

In each of the three cases, we define the exceptional set $\cE$ and get an upper bound for $\ell_{\sigma}$ for $\sigma\not\in \cE$.
We also need to control the cardinality of $\cE$, of course.

\paragraph{1st case: $b\ge \frac{1}{12}$.}
In this case $w_{n}\ge \frac{b}{2}$. Let $A=\{k\colon w_{k}\ge \frac{b}{4}\}$ and observe that
$|A|\le \frac{16}{b^{2}}$ since  $\sum w_{i}^{2}=1$.

Now fix any $\sigma\in\Sym_{n-1}$. If, for some $k\in\{1, 2, \ldots , n-1\}$, $\sigma_{k}\not\in A$, then,
by~\eqref{eq:weightsincreaseineachgroup}, we see that $\text{wt}[\pi^{(k+1)}]-\text{wt}[\pi^{(k)}] > \frac{b}{4}$.
Recall that bad permutations are those for which $\text{wt}[\pi]$ is in $[Ln-1,Ln+1]$, an interval of length $2$.
This shows that
\[
\Bigl( \ell_{\sigma}-1-\frac{16}{b^{2}}\Bigr)\, \frac{b}{4} \le w_{k_{\sigma}+\ell_{\sigma}-1}-w_{k_{\sigma}}\le 2\,.
\]
Hence $\ell_{\sigma} \le \frac{16}{b^{2}}+\frac{8}{b}+1$ for all $\sigma\in \Sym_{n-1}$.
Set $C=\frac{16}{(1/12)^{2}}+\frac{8}{(1/12)}+1$. Let $\cE=\emptyset$ (no need in exceptional permutations in this
case). Then,
\begin{align}\label{eq:exceptionalset1}
|\cE| &= 0 \text{ and } \ell_{\sigma}\le C \text{ for all } \sigma \in \Sym_{n}.
\end{align}

\paragraph{2nd case: $n^{-\frac{1}{3}}\le b\le \frac{1}{12}$.}
Fix $\sigma$ and let $I_{\sigma}=\{\sigma (j)\colon k_{\sigma}\le j < k_{\sigma}+\ell_{\sigma}-1\}$
and $A=\{k\colon w_{k}\ge \frac{b}{4}\}$. Exactly as in the previous case, $|A|\le \frac{16}{b^{2}}$ (which may be rather large in the present case,
hence further arguments below) and
\[
( \ell_{\sigma}-1-|A\cap I_{\sigma}| )\,\frac{b}{4} \le  2\,.
\]
Therefore, if $\ell_{\sigma}\ge \frac{10}{b}$, then
$\bigl| A\cap \{\sigma (j)\colon k_{\sigma}\le j < k_{\sigma}+\frac{10}{b}-1\} \bigr| \ge \frac{1}{b}$.
Let $m=\lceil \frac{1}{b}\rceil$ and define
$$
\cE=\bigl\{ \sigma\colon |A\cap \{\sigma (j)\colon k\le j\le k+10m-1\}| \ge m\text{ for some }k\bigr\}.
$$
Then,  $\ell_{\sigma}\le \frac{10}{b}$ for $\sigma \not\in \cE$.

\medskip
We now need to bound the cardinality of the exceptional set $\cE$.
For any $\sigma\in \cE$, we associate a $(2m+1)$-tuple $(k, j_{1}, \ldots , j_{m}, a_{1}, \ldots , a_{m})$,
where $1\le k\le j_{1}< \ldots <j_{m}< k+10m-1\le n$ and  $a_{1},\ldots ,a_{m}$ are elements of $A$ such that
$\sigma(j_{1})=a_{1},\, \ldots \,, \sigma(j_{m})=a_{m}$. By definition of $\cE$, such a tuple exists, and if there is more than one, make an arbitrary choice to fix
one.

Now we get an upper bound for the number of distinct $(2m+1)$-tuples that can arise in this manner. Firstly, the number of choices for $k$ is less than $n$,
and having chosen $k$, the numbers $j_{1},\, \ldots\,, j_{m}$ may be chosen in less than $2^{10m}$ ways (any subset of $\{k, k+1,\, \ldots\,, k+10m-1\}$)
and after that choice, $a_{1},\, \ldots\,, a_{m}$ may be chosen in at most $|A|^{m}$ ways.

Finally, any given $(2m+1)$-tuple can come from at most $(n-m)!$ permutations, since the values of $\sigma(j_{1}),\, \ldots\,, \sigma(j_{m})$ are fixed.
Thus,
\[
\frac{1}{(n-1)!}\, |\cE| \le \frac{(n-m)!}{(n-1)!}\, n\, 2^{10m}\, \Bigl( \frac{16}{b^{2}}\Bigr)^{m}
\le \frac{n}{(n/2)^{\frac{1}{b}-1}}\, \Bigl( \frac{C}{b^{2}} \Bigr)^{1/b}  \le n^2 \Bigl( \frac{C}{nb^2}\Bigr)^{1/b}
\]
since $|A|\le \frac{16}{b^{2}}$ and $n-m\ge \frac{n}{2}$ (as $m=\lceil \frac{1}{b}\rceil \le 2\sqrt{n}$).
If $b\ge n^{-1/3}$ (in fact we may go up to $b\ge n^{-\frac12 + \delta}$ for any $\delta>0$), it is easy to see that the above quantity is bounded by
\[
n^2 \Bigl( \frac{C}{n^{-2/3}\cdot n}\Bigr)^{12} \le \frac{C}{n^2}\,.
\]

These manipulations are all valid for $n\ge n_{0}$ for some fixed $n_{0}$. Thus,
\begin{align}\label{eq:exceptionalset2}
\frac{1}{(n-1)!}\, |\cE| \le \frac{C}{n^{2}} \quad \text{ and }\ \ell_{\sigma}\le \frac{C}{b} \text{ for all }\sigma \not\in \cE\,,
\end{align}
where the constant $C$ can take care of all the cases when $n<n_{0}$.

\paragraph{3rd case: $\frac{1}{\sqrt{n}}\le b\le n^{-\frac{1}{3}}$.} Setting $A=\{k\colon w_{n}-w_{k}<\frac{b}{2}\}$ we see that $|A|\le \frac{n}{2}$
(because of the way we chose  $w_{n}$). Fix $\sigma$ and let $I_{\sigma}=\{\sigma (j)\colon k_{\sigma}\le j < k_{\sigma}+\ell_{\sigma}-1\}$. Again
\[
(\ell_{\sigma}-1-|A\cap I_{\sigma}|)\,\frac{b}{2}\le 2.
\]
Hence if $\ell_{\sigma}\ge \frac{12}{b}$, then $|A\cap I_{\sigma}|\ge \frac{7}{b}$.  Define
$$
\cE=\Bigl\{ \sigma\colon \bigl| A\cap \bigl\{ \sigma (j)\colon k\le j < k + \frac{12}b \bigr\} \bigr| \ge \frac7b \text{ for some }k \Bigr\}.
$$
Then,  $\ell_{\sigma}\le \frac{12}{b}$ for $\sigma \not\in \cE$. We want to bound the cardinality of $\cE$.

\medskip
Let $m=\lceil \frac{1}{b}\rceil$ and let $\cE=\bigcup_{k} \cE_{k}$ where
$$\cE_{k}=\bigl\{ \sigma\colon \bigl| A\cap \{\sigma (j)\colon k\le j\le k+12m-1\} \bigr| \ge 7m \bigr\}\,.$$
Fix $k\le n-12m+1$. The quantity $ \frac1{(n-1)!}\, \bigl| \cE_k \bigr| $ has the following interpretation.
Suppose we have a basket with $n-1$ different balls labeled by $\{1, 2, \, \ldots \, n-1\}$,
$|A|\le \frac12 n$ of these balls are black, while the rest are white.
We take, at random, $12m$ different balls (without returning them
to the basket). Then, $ \frac1{(n-1)!}\, \bigl| \cE_k \bigr| $ is the probability
that at least $7m$ of these $12m$ balls will be black. Whence, using e.g. Stirling's formula, for $n\ge n_0$,
\[
\frac1{(n-1)!}\, \bigl| \cE_k \bigr| \le e^{-cm}\,.
\]

This is the bound for fixed $k$. Add over $k$ to see that $\frac{1}{(n-1)!}\,|\cE|\le ne^{-cm}$. Since $m\ge n^{1/3}$, this gives
\begin{align}\label{eq:exceptionalset3}
\frac{1}{(n-1)!}\, |\cE| \le \frac{C}{n^{2}} \text{ and }\ell_{\sigma}\le \frac{C}{b} \text{ for all }\sigma \not\in \cE\,.
\end{align}
Again, the constant $C$ is adjusted so that the above estimates are also valid for $n<n_{0}$. This completes the third case.

\subsubsection{Completing the proof of Lemma~\ref{lem:mainanticoncentration}}

To finish the proof, recall that if $\sigma\not\in \cA$, then  $\BB\cap \widehat{\Sym}_{\sigma}$ is empty. Consequently,
\[
\frac{1}{n!}|\BB| \le \frac{1}{n!}\, \sum_{\sigma\in \cA} \ell_{\sigma}
\le \frac{1}{n!}\, \frac{C}{b}\, |\cA| + \frac{1}{n!}\, n\, |\cA\cap \cE|\,,
\]
by applying the bound $\ell_{\sigma} \le \frac{C}b$ for $\sigma\notin \cE$ as given in~\eqref{eq:exceptionalset1}, \eqref{eq:exceptionalset2}
and~\eqref{eq:exceptionalset3} and the trivial bound $\ell_{\sigma}\le n$ for $\sigma\in \cE$.
Note that $|\cA\cap \cE|\le |\cA|\wedge |\cE|\le \sqrt{|\cA|}\sqrt{|\cE|}$.
>From the bounds~\eqref{eq:exceptionalset1}, \eqref{eq:exceptionalset2} and~\eqref{eq:exceptionalset3}
on the cardinality of $\cE$
and the bound~\eqref{eq:firstcrudebound} on the cardinality of $\cA$, we get
\[
\frac{1}{n!}\, |\BB| \le \frac{C}{n}\,e^{-c|L|} + \frac{1}{(n-1)!}\, \sqrt{(n-1)!\frac{C}{n^{2}}} \cdot \sqrt{(n-1)!\, Cbe^{-c|L|}}
\le \frac{C'}{n}\,e^{-c'|L|}.
\]
This completes the proof of Lemma~\ref{lem:mainanticoncentration}.
\hfill $\Box$

\section{The proof of Lemma~\ref{lem:unimodal}}\label{sec:proofofprobestimates}

We conclude the paper by proving Lemma~\ref{lem:unimodal}, which says that the density of the distribution of the
sum $X=\sum_{i=1}^{n}w_{i}U_{i}$ (where $U_i$ are i.i.d. random variables with uniform distribution on $\bigl[ -\frac12, \frac12 \bigr]$) has the bound $p_{X}(t)\le
Ce^{-c |t|}$.
We will give two proofs of this fact. The first proof is based on properties of logarithmically
concave distributions.
The second proof combines the classical Bernstein-Hoeffding estimate with
a simple argument based on the Fourier transform.
Note that the second proof yields a somewhat stronger
conclusion that $p_X(t) \le Ce^{-c t^{2}}$. This, in turn,
improves the factor $e^{-c|L|}$ in Lemmas~$4'$ and~\ref{lem:mainanticoncentration} to $e^{-cL^2}$.

\subsection{Log-concavity}
A random vector in $\bR^{d}$ having density $p(\cdot)$ is said to be log-concave if the function $\log p\colon \bR^{d}\to \bR\cup\{-\infty\}$
is concave. For our purposes, it suffices to know the following two basic classes of examples and a general property of log-concave densities
in one dimension, given below in Lemma~\ref{lem:logconcave}.
\begin{enumerate}
\item If the random vector $X$ is uniformly distributed on a compact convex set $K$,
then $X$ is log-concave. Indeed, if $p$ is the density of $X$, then
\[
\log p(x) = \begin{cases}
\log(1/\text{vol}(K)) &\text{ if }x\in K, \\
-\infty &\text{ if }x\not\in K,
\end{cases}
\]
which is easily seen to be concave.
\item If $X$ is as above (uniform on a convex set in $\bR^{d}$), and $u\in \bR^{d}$ is any fixed vector, then $\<u,X\>$ is log-concave in one dimension. This is
    not an obvious fact, but is a consequence of the Pr\'{e}kopa-Leindler inequality~\cite{Prekopa, Gardner}.
\end{enumerate}
As a consequence, if $V$ is uniformly distributed on $[0,1]^{n}$ (a convex set) then, the scalar product $\<w,V\>$ is a log-concave random variable
in one dimension, for any $w\in \bR^{n}$.

Here is the one key (and well-known to experts) property of log-concave random variables that we need.
\begin{lemma}\label{lem:logconcave} Let $X$ be a real-valued, symmetric, log-concave random variable with unit variance. Then the density $p(\cdot)$ of $X$ satisfies
$p(t)\le Ce^{-c|t|}$ for all $t$ for some constants $C,c$.
\end{lemma}
The lemma remains valid if we drop the assumption of symmetry and instead assume that $X$ has zero mean, but the proof would be a bit longer.
Since we only apply this to symmetric random variables, we state only this weaker version.

\begin{proof}[Proof of Lemma~\ref{lem:logconcave}] Since $\log p$ is concave and symmetric, it is non-increasing on $(0,\infty)$ and non-decreasing on $(-\infty,
0)$. Define $B=\inf\{t>0\colon p(t)\le \frac12 p(0)\}$. Then $0<B <\infty$.  Further,  log-concavity shows that $p(kB)^{1/k} p(0)^{(k-1)/k}\le p(B) \le
\frac12 p(0)$, whence, $p(kB)\le p(0)2^{-k}$.  By the unimodality of $p$, we now have the bounds
\begin{align*}
p(t)&\ge \frac12\, p(0) \text{ if }t\in (0,B),\\
p(t)&\le p(0)2^{-k} \text{ if }t\in (kB,(k+1)B), \; k\ge 0.
\end{align*}
>From these bounds, we get
\begin{align}
\frac{1}{2}\, B p(0) &\le \int_{0}^{\infty}p(t)\, {\rm d}t \le 2B p(0), \label{eq:boundsforbp01} \\
\frac{1}{6}\, B^{3}p(0) &\le \int_{0}^{\infty}t^{2}p(t)\, {\rm d}t \le 12B^{3}p(0). \label{eq:boundsforbp02}
\end{align}
In the last inequality, we used the fact that $\sum_{k=0}^{\infty}2^{-k}(k+1)^{2}=12$ (any constant upper bound would suffice).
But, $\int_{0}^{\infty} p(t)\, {\rm d}t = \frac12 $ and $\int_{0}^{\infty}t^{2}p(t)\, {\rm d}t=\frac12\, \text{Var}[X]=\frac12$.
Combining with~\eqref{eq:boundsforbp01} and~\eqref{eq:boundsforbp02},
we see  that $c \le p(0)\le C$ and $c\le B\le C$ for some numerical constants $c,C$.

Thus, for $t\in [-B,B]$, we have $p(t)\le C$. Also, for $t>kB$ we have $p(t)\le p(0)2^{-k}$
which gives the exponential upper bound $p(t)\le C\, e^{-ct}$.
\end{proof}

\subsection{Proof of Lemma~\ref{lem:unimodal}}
As $U=(U_{1},\ldots ,U_{n})$ is uniformly distributed on $\bigl[ -\frac12 ,\frac12 \bigr]^{n}$,
we see that $X=\sum_{i=1}^{n}w_{i}U_{i}$ is a symmetric, log-concave random variable.
Its variance is $\frac{1}{12}\sum_{i=1}^{n}w_{i}^{2}$.
If $\sum_{i=1}^{n}w_{i}^{2}=1$, then, by Lemma~\ref{lem:logconcave}, the density of $X$ is bounded by $Ce^{-c |t|}$.
This proves the first part of the lemma.

\medskip
If $\|w\|^{2}=\sum_{i=1}^{n}w_{i}^{2}<1$, then let $Y=\frac{1}{\|w\|}X$.
By the first part, the density of $Y$ is bounded by $Ke^{-\kappa |t|}$.
The density of $X$ is given by
\[
p_{X}(t) = p_{Y}(t/\|w\|)\,\|w\|^{-1} \le \frac{C}{\|w\|}\, e^{-c |t|/\|w\|}\,.
\]
Since the function $x\to xe^{-c |t|x/2}$ is bounded by $\frac{2}{c |t|}$, we get the bound
\[
p_{X}(t) \le \frac{2C}{c |t|}\, e^{-c |t|/(2\|w\|)}\,.
\]
For $|t|\ge  1$ and $\|w\|\le 1$, this is less than or equal to $C'e^{-c'|t|}$.
\hfill $\Box$

\subsection{Another proof of (an improved version of)
	Lemma~\ref{lem:unimodal}}\label{sec:improvedinequality}

First, we recall the following classical inequality which goes back to Bernstein and Hoeffding.

\subsubsection{The Bernstein-Hoeffding inequality}\label{subsubsect:BH}
{\em If $X_{1},\ldots ,X_{n}$ are independent random variables with zero mean,  and such that $|X_{i}|\le a_{i}$ a.s., where $a_{i}$ are non-negative numbers.
Then, for any $t>0$,
\begin{align}\label{eq:hoeffding}
\bP\Bigl\{ \sum_{i=1}^{n}X_{k}\ge t \Bigr\} \le \exp \Bigl\{ -\frac12\, \frac{t^{2}}{\sum_{i=1}^{n}a_{i}^{2}} \Bigr\}\,.
\end{align}
}

\subsubsection{A bound for $p_X(t)$ when $|t|\ge 1$}
Now consider $X_{i}=w_{i}U_{i}$ as in Lemma~\ref{lem:unimodal}.
Clearly $X_{i}$ are independent  and $|X_{i}|\le \frac{1}{2}\, |w_{i}|$. Therefore, by~\eqref{eq:hoeffding}, we get
\[
\bP\{ X\ge t \} \le \exp\Bigl\{ -\frac{1}{2}\, \frac{t^{2}}{\sum_{i=1}^{n}(w_{i}/2)^{2}}\Bigr\}
\le e^{-2t^{2}}
\]
since $\sum_{i=1}^{n}w_{i}^{2}\le 1$.
To get a bound on the density from the bound on the tail, we observe that $X$ is a symmetric (i.e., $p_{X}(t)=p_{X}(-t)$) and unimodal\footnote{
The unimodality may be proved easily by induction on the number of summands.}
(that is,  $p_{X}$ is non-decreasing on $(-\infty,0]$ and non-increasing on $[0,\infty)$).
Therefore,  for any $t>0$ we get
\[
\frac{t}{2}\, p_{X}(t) \le \int_{t/2}^{t}p_{X}(s)\, {\rm d}s
\le \bP\{ X\ge t/2 \} \le e^{-t^2/2}\,.
\]
For $t>1$, this gives the bound $p_{X}(t)\le 2e^{-t^2/2}$.
By symmetry, the same bound holds for $t<-1$.
This is what we set out to prove when $|t|\ge 1$. \hfill $\Box$

\subsubsection{A bound for $p_X(t)$ when $|t|\le 1$}

Here, $X=\sum_i w_i U_i $, where $U_i$ are i.i.d. random variables uniformly distributed
on $\bigl[ -\frac12, \frac12 \bigr]$, and $\sum_i w_i^2 = 1$. We will show that, for $|t|\le 1$,
$ p_X(t) \le C$.

\medskip
Let $\delta$ be a small positive parameter, to be chosen near the end of the proof. If, for some $i$,
$ |w_i|\ge \delta $, then we use the estimate
\[
p_X(t) \le \frac1{\delta}\,, \qquad |t|\ge 1\,,
\]
which follows from the fact that the sup-norm of the density cannot increase
under convolution.
In what follows, we assume that $|w_i| < \delta$ for all $i\ge 1$.

The characteristic function of $X$ equals
\[
\bE \bigl[ e^{{\rm i}\lambda X } \bigr] = \prod_{i\ge 1} \frac{\sin (\lambda w_i)}{\lambda w_i}\,,
\]
so we need to show that
\[
\int_0^\infty \prod_{i \ge 1} \Bigl| \frac{\sin (\lambda w_i)}{\lambda w_i} \Bigr|\, {\rm d}\lambda \le C,
\]
provided that $ \sum_i w_i^2 = 1 $ and that $|w_i| < \delta$. We will be using that
\[
\Bigl| \frac{\sin\xi}{\xi} \Bigr| \le
\begin{cases}
e^{-c\xi^2}, &|\xi|\le 2, \\
\frac1{|\xi|}\,, &|\xi|\ge 1\,.
\end{cases}
\]

Let $I_j = [2^j, 2^{j+1}]$, $ j \ge 1 $, and let
$J_k = \bigl\{ i\colon 2^{-k} < |w_i| \le 2^{1-k} \bigr\}$, $k\ge 1$.
Then
\[
\frac14 \le \sum_{k\ge 1} |J_k| 2^{-2k} \le 1\,.
\]
For $\lambda\in I_j$, $i\in J_k$,
we have $2^{j-k} \le |\lambda w_i | \le 2^{2+j-k}$, whence,
\[
\Bigl| \frac{\sin (\lambda w_i)}{\lambda w_i} \Bigr| \le
\begin{cases}
e^{-c 2^{2(j-k)}}, & k \ge j, \\
2^{k-j}, & 1\le k \le j-1\,.
\end{cases}
\]
Thus,
\begin{multline}
\int_{I_j}\, \prod_{i \ge 1} \Bigl| \frac{\sin (\lambda w_i)}{\lambda w_i} \Bigr|\, {\rm d}\lambda
\le 2^j\, \prod_{1\le k \le j-1} 2^{(k-j)|J_k|} \,
\prod_{k\ge j} e^{-c2^{2(j-k)} |J_k|} \\
= 2^j\, \exp\Bigl[ -\log 2\, \sum_{1\le k \le j-1} (j-k) 2^{2k}\, 2^{-2k} |J_k|
- c 2^{2j} \sum_{k\ge j} 2^{-2k} |J_k| \Bigr]\,. \label{eq:density}
\end{multline}
Now, we fix sufficiently large positive integers $j_0$ and $k_0$ so that, for $j\ge j_0$, we have $c 2^{2j} \ge 8\log 2$
(here, $c$ is the
constant in the exponent on the right-hand side of~\eqref{eq:density}), and that, for
$j\ge j_0$, $k_0 \le k \le j-1$, we have $(j-k)2^{2k} \ge 8j$. Then, we take $\delta = 2^{-k_0}$. This choice guarantees
that $J_k = \emptyset$ for $1\le k < k_0$.
Then, for $j\ge j_0$,
\[
\text{right-hand side of~\eqref{eq:density}}
\le 2^j\, \exp\Bigl[ -(8\log 2)j\, \sum_{k\ge k_0} 2^{-2k}|J_k| \Bigr] \le 2^{-j}\,.
\]
Therefore,
\[
\int_0^\infty \, \prod_{i \ge 1} \Bigl| \frac{\sin (\lambda w_i)}{\lambda w_i} \Bigr|\, {\rm d}\lambda
= \Bigl(\, \int_0^{2^{j_0}} + \int_{2^{j_0}}^\infty \,\Bigr)\ \ldots \
\le 2^{j_0} + \sum_{j\ge j_0} 2^{-j} < 2^{j_0} +1,
\]
completing the proof. \hfill $\Box$

\section{Proof of Claim~\ref{lem:varianceofwdotvpi}}\label{sect:proof-of-claim9}

Let $V_{1},\ldots ,V_{n}$ be i.i.d. uniform random variables,
and $V_{(1)}<V_{(2)}<\ldots <V_{(n)}$ be their order statistics
(i.e., $V_{(j)}$ is the $j$th minimum of
the $V_{i}$s). Since the distribution of $(V_{(1)},V_{(2)},\ldots ,V_{(n)})$ is the same
as that of $((V_{\sigma})_{\sigma (1)},\ldots ,(V_{\sigma})_{\sigma (n)})$,
$\langle w,V_{\sigma} \rangle$ has the same distribution
as $\sum_{j=1}^{n}w_{\sigma(j)}V_{(j)}$. Put $X_j=V_{(j)}-V_{(j-1)}$,
$1\le j \le n+1$, with the convention that $V_{(0)}=0$
and $V_{(n+1)}=1$.
Thus, the lemma will follow if we consider
\[
Y = \sum_{j=1}^{n}w_{\sigma (j)}V_{(j)}
= \sum_{j=1}^n \alpha_j X_j
\]
with $\alpha_j = w_{\sigma (j)}+ \ldots + w_{\sigma (n)}$,
and show that $\operatorname{Var}[Y]\le \frac{1}{(n+1)(n+2)}
\sum_{j=1}^{n}\alpha_{j}^{2}$.

To show this,
we will need one property of the distribution of
the vector $(X_{1},\ldots ,X_{n+1})$,
that of exchangeability: if
the $X_{i}$s are permuted, the
resulting vector has the same distribution as $(X_{1},\ldots ,X_{n+1})$.
To see this, drop one of the $X_{i}$s and write the density of the remaining $n$ random variables with respect to Lebesgue measure on $\bR^{n}$.
A change of variables shows that the density is uniform on
$\{(t_{1},\ldots, t_{n})\colon t_{i}\ge 0, \ \sum_{i=1}^{n}t_{i}< 1\}$. Then, exchangeability is clear.

>From exchangeability, we see that $\bE [X_{i}]=\bE [X_{1}]$, $\bE [X_{i}^{2}]=\bE [X_{1}^{2}]$ and $\bE [X_{i}X_{j}]=\bE [X_{1}X_{2}]$ for all $i\not=j$. Further,
$X_{1}+\ldots +X_{n}=1$. Therefore, $\bE[X_{1}]=\frac{1}{n+1}$ and
\[
(n+1)\bE[X_{1}^{2}]+n(n+1)\bE[X_{1}X_{2}]=1
\]
since the left hand side is $\bE[(X_{1}+\ldots +X_{n+1})^{2}]$.

Since $X_{1}=V_{(1)}$ is the minimum of $n$ uniform random variables, we easily calculate that $\bE[X_{1}^{2}]=\frac{2}{(n+1)(n+2)}$. Then we also get
$\bE[X_{1}X_{2}]=\frac{1}{(n+1)(n+2)}$.

Thus, we get $\bE[Y]=\frac{1}{n+1}\sum_{j=1}^{n}\alpha_{j}$, and
\begin{align*}
\bE[Y^{2}]  &= \frac{2}{(n+1)(n+2)}\sum_{j=1}^{n}\alpha_{j}^{2}+\frac{1}{(n+1)(n+2)}\sum_{i\not= j}\alpha_{i}\alpha_{j} \\
&= \frac{1}{(n+1)(n+2)} \sum_{j=1}^{n} \alpha_{j}^{2} + \frac{1}{(n+1)(n+2)} \Bigl( \sum_{j=1}^{n} \alpha_{j} \Bigr)^{2}.
\end{align*}
Finally,
\begin{align*}
\operatorname{Var}[ \langle w,V_{\sigma} \rangle ] &= \bE[Y^{2}]-(\bE[Y])^{2} \\
&= \frac{1}{(n+1)(n+2)}\sum_{j=1}^{n}\alpha_{j}^{2} -
\frac{1}{(n+1)^2(n+2)} \Bigl( \sum_{j=1}^{n}\alpha_{j} \Bigr)^{2}
\end{align*}
which completes the proof. \qed

\medskip\noindent{\bf Remark.} The distribution of the $X_i$s appearing in the proof of
Claim \ref{lem:varianceofwdotvpi} is called the Dirichlet distribution with parameters
$n+1$ and $(1,\ldots,1)$.

\end{document}